\documentclass[a4paper,10pt]{amsart}
\usepackage{amssymb,amsmath,amsfonts,amsthm}
\usepackage[dvips]{graphicx}
\usepackage{psfrag}

\theoremstyle{plain}
\newtheorem{main}{Theorem}

\newtheorem{maincor}[main]{Corollary}
\newtheorem{theorem}{Theorem}[section]
\newtheorem{lemma}[theorem]{Lemma}
\newtheorem{proposition}[theorem]{Proposition}
\newtheorem{corollary}[theorem]{Corollary}
\theoremstyle{remark}
\newtheorem{remark}[theorem]{Remark}

\newcommand{\quand}{\quad\text{and}\quad}
\newcommand{\Leb}{\operatorname{vol}}

\newcommand{\C}{\operatorname{C}}

\newcommand{\tf}{\tilde{f}}
\newcommand{\tg}{\tilde{g}}
\newcommand{\tx}{\tilde{x}}

\newcommand{\tphi}{\tilde{\phi}}
\newcommand{\diam}{\operatorname{diam}}

           \def\ea{\end{array}}
          \def\ec{\end{center}}
     \def\ed{\end{description}}
        \def\ee{\end{equation}}
       \def\eea{\end{eqnarray}}
     \def\eeaa{\end{eqnarray*}}
 \def\et{\end{thebibliography}}

\def\Orb{{\rm Orb}}

\def\real{\mathbb R}

\def\supp{\operatorname{supp}}

\def\cA{{\mathcal A}}

\def\cD{{\mathcal D}}

\def\cK{{\mathcal K}}

\def\cU{{\mathcal U}}
\def\cV{{\mathcal V}}

\def\cB{{\mathcal B}}

\def\cF{{\mathcal F}}

\def\id{\operatorname{id}}
\def\length{\operatorname{length}}

\def\vep{\varepsilon}

\def\TT{{\mathbb T}}
\def\RR{{\mathbb R}}

\def\ZZ{{\mathbb Z}}

\title[Measure-theoretical properties of center foliations]{Measure-theoretical properties of center foliations}
\author{Marcelo Viana and Jiagang Yang}
\date{\today}
\thanks{M.V. and J.Y. were partially supported by CNPq, FAPERJ, and PRONEX}

\address{IMPA, Est. D. Castorina 110, 22460-320 Rio de Janeiro, Brazil}
\email{viana\@@impa.br}

\address{Departamento de Geometria, Instituto de Matem\'atica e Estat\'\i stica, Universidade Federal Fluminense, Niter\'oi, Brazil}
\email{yangjg\@@impa.br}

\begin{document}

\begin{abstract}
Center foliations of partially hyperbolic diffeomorphisms may exhibit pathological behavior from a measure-theoretical viewpoint:
quite often, the disintegration of the ambient volume measure along the center leaves consists of atomic
measures. We add to this theory by constructing stable examples for which the disintegration is singular without being atomic.
In the context of diffeomorphisms with mostly contracting center direction, for which upper leafwise absolute continuity is
known to hold, we provide examples where the center foliation is not lower leafwise absolutely continuous.
\end{abstract}

\maketitle

\tableofcontents

\setcounter{tocdepth}{1} \tableofcontents

\section{Introduction}
As is the case for many other developments in Dynamics over the last half century, the subject of this paper goes back to the
work of Dmitry Viktorovich Anosov.

Anosov's remarkable proof~\cite{An67} that the geodesic flow on any manifold with negative curvature is ergodic introduced two
major ingredients in Dynamics. The first one was the observation that those geodesic flows are \emph{hyperbolic},
which implies that they carry certain invariant -- stable and unstable -- foliations.
The second one is the proof that those foliations, while not being smooth, are still regular enough (\emph{absolute continuity})
that a version of the Hopf ergodicity argument can be applied.

Here we consider maps exhibiting a weaker (partial) form of hyperbolicity and we want to study the properties of the invariant
\emph{center} foliations. Before stating our results in precise terms, let us briefly outline how this field has evolved.

\subsection{Partial hyperbolicity}
Recall that a diffeomorphism $f:M\to M$ on a compact manifold is called an \emph{Anosov diffeomorphism} (or \emph{globally hyperbolic})
if there exists a decomposition $TM = E^s \oplus E^u$ of the tangent bundle $TM$ into two continuous sub-bundles $x\mapsto E^s_x$
and $x\mapsto E^u_x$ such that
\begin{itemize}
\item both $E^s$ and $E^u$ are invariant under the derivative $Df$ and
\item $Df \mid E^s$ is a uniform contraction and $Df\mid E^u$ is a uniform expansion.
\end{itemize}
Such systems form an open (possibly, empty) subset of the space of $C^r$ diffeomorphisms of $M$, for any $r\ge 1$.

A distinctive feature of Anosov diffeomorphisms is that they admit invariant foliations $\cF^s$ and $\cF^u$ that are tangent to
the sub-bundles $E^s$ and $E^u$ at every point. Consequently, the leaves of $\cF^s$ are contracted by the forward iterates,
whereas the leaves of $\cF^u$ are contracted by backward iterates of $f$. The leaves are smooth immersed sub-manifolds but,
in general, these foliations are not differentiable, that is, they can not be ``straightened'' by means of $C^1$ local charts.
However, after Anosov, Sinai~\cite{An67,AS67}, we know that they do have a crucial differentiability property, called
\emph{absolute continuity}: assuming the derivative $Df$ is H\"older continuous, the holonomy maps of both $\cF^s$ and $\cF^u$
map zero Lebesgue measure sets to zero Lebesgue sets. Indeed, this fact lies at the heart of Anosov's proof that the geodesic
flow on manifolds with negative curvature is ergodic.

By the early 1970's, Brin, Pesin~\cite{BP74} were proposing to extend the class of Anosov diffeomorphisms to what they
called \emph{partially hyperbolic} diffeomorphisms. A similar proposal was made by Pugh, Shub~\cite{PSh72} independently
and at about the same time. By partially hyperbolic, we mean in this paper\footnote{Brin, Pesin used a stronger definition
that is sometimes called \emph{absolute partial hyperbolicity}. See Hirsch, Pugh, Shub~\cite[pages~3--5]{HPS77}.}
that there exists a decomposition $TM = E^{ss} \oplus E^c \oplus E^{uu}$ of the tangent bundle $TM$ into three continuous
sub-bundles $x\mapsto E^{ss}_x$ and $x\mapsto E^c_x$ and $x\mapsto E^{uu}_x$ such that
\begin{itemize}
\item[(i)] all three sub-bundles $E^{ss}$ and $E^c$ and $E^{uu}$ are invariant under the derivative $Df$ and
\item[(ii)] $Df \mid E^{ss}$ is a uniform contraction, $Df\mid E^{uu}$ is a uniform expansion and
\item[(iii)] $Df \mid E^c$ lies in between them:
for some choice of a Riemannian metric on $M$ (see Gourmelon~\cite{Go07}), one has
$$
\frac{\|Df(x)v^s\|}{\|Df(x)v^c\|} \le \frac 12
\quand
\frac{\|Df(x)v^c\|}{\|Df(x)v^u\|} \le \frac 12
$$
for any unit vectors $v^s\in E^{ss}$ and $v^c\in E^c$ and $v^u\in E^{uu}$ and any $x\in M$ (we say that $E^c$ \emph{dominates} $E^{ss}$
and $E^{uu}$ \emph{dominates} $E^c$, respectively.)
\end{itemize}

Again, partially hyperbolic diffeomorphisms form an open subset of the space of $C^r$ diffeomorphisms of $M$, for any $r\ge 1$.
The class of manifolds for which this set is non-empty is far from being completely understood.

\subsection{Stable and unstable foliations}

Part of what was said before about Anosov diffeomorphisms extends to this class. Namely, the \emph{strong-stable sub-bundle} $E^{ss}$
and the \emph{strong-unstable sub-bundle} $E^{uu}$ are still uniquely integrable, that is, there are unique foliations
$\cF^{ss}$ and $\cF^{uu}$ whose leaves are smooth immersed sub-manifolds of $M$ tangent to $E^{ss}$ and $E^{uu}$, respectively, at every point.
Moreover, these so-called \emph{strong-stable foliation} and \emph{strong-unstable foliation} are still absolutely continuous.
This fact plays a key role in the ergodic and geometric theory of such systems (see Pugh, Shub~\cite{PSh97} and Burns, Wilkinson~\cite{BW10},
for example).
\subsection{Center foliations: existence and (non-)absolute continuity}

The situation for the \emph{center sub-bundle} $E^c$ is a lot more complicated. To begin with, $E^c$ need not be integrable, that
is, there may be no foliation with smooth leaves tangent to $E^c$ at every point. The first example was probably due to
Smale~\cite{Sm67}, see Wilkinson~\cite{Wil98} and Pesin~\cite{Pes04}; other constructions, with interesting additional features,
were proposed by Hammerlindl~\cite{Ham14} and Hertz, Hertz, Ures~\cite{HHUcoh}. This later paper also shows that even when the
center sub-bundle is integrable it may fail to be \emph{uniquely integrable}, that is, curves tangent to $E^c$ may not be
contained in a unique leaf of the integral foliation (\emph{center foliation}).

Notwithstanding, there are also many robust examples of partially hyperbolic diffeomorphisms with uniquely integrable center
sub-bundle. The simplest construction goes as follows. Start with a hyperbolic torus automorphism $A:\TT^3\to\TT^3$
(a similar construction can be carried out in any dimension) with eigenvalues
\begin{equation}\label{eq.Anosov}
\lambda_1 < 1 < \lambda_2 < \lambda_3
\end{equation}
and corresponding eigenspaces $E_1$, $E_2$, $E_3$. $A$ is an Anosov diffeomorphism, of course, and then so is any diffeomorphism $f$
in a $C^1$ neighborhood. However, $A$ may also be viewed as a partially hyperbolic diffeomorphism with invariant sub-bundles $E^{ss}_x=E^1$ and
$E^c_x=E^2$ and $E^{uu}_x=E^3$. Then every $f$ in a $C^1$ neighborhood is also partially hyperbolic. It follows from general
results in \cite{HPS77} that the center (or ``weakly expanding'') bundle $E^c$ of $f$ is uniquely integrable in this case.
Actually, a result of Potrie~\cite{Pot15} implies that the center sub-bundle is integrable for every partially
hyperbolic diffeomorphism in the isotopy class of $A$. Moreover, if one assumes the (stronger) absolute form of partial
hyperbolicity that we alluded to before, it follows from Brin, Burago, Ivanov~\cite{BBI09} that the center sub-bundle is
uniquely integrable for any partially hyperbolic diffeomorphism of $\TT^3$.

Thus, the question naturally arises whether such center foliations are still absolutely continuous. In fact, this question
was first raised by A. Katok in the 1980's, especially for Anosov diffeomorphisms in $\TT^3$ as introduced in the
previous paragraph. Katok also obtained the first example of a center foliation (for a non-invertible map) which is not
absolutely continuous. Indeed, this foliation (see Milnor~\cite{Mi97a}) is such that some full volume subset intersects
each leaf in not more than one point.

Shub, Wilkinson~\cite{SW00} constructed partially hyperbolic, stably ergodic (with respect to volume) diffeomorphisms
whose center leaves are circles and whose center Lyapunov exponent is non-zero, and they observed that for such maps
the center foliation can not be absolutely continuous. Indeed, in a related setting, Ruelle, Wilkinson~\cite{RW01} observed
that the center foliation has \emph{atomic disintegration}: the Rokhlin conditional measures of the volume measure
along the leaves are supported on finitely many orbits. That is the case also in Katok's construction, as observed before,
but it should be noted that in Katok's example the center Lyapunov exponent vanishes. An extension of these results to
diffeomorphisms with higher-dimensional compact center leaves was due to Hirayama, Pesin~\cite{HiP07}.

As a matter of fact, for a large class of partially hyperbolic, volume-preserving diffeomorphisms with one-dimensional
center leaves one has a sharp dichotomy \emph{atomic disintegration} vs. \emph{Lebesgue disintegration}:
the conditional measures of the volume measure along center leaves are either purely atomic or equivalent to the Lebesgue
measure; in this latter case, we also speak of \emph{leafwise absolute continuity}.
This was observed by Avila, Viana, Wilkinson, in two main situations:
\begin{itemize}
\item maps fixing their center leaves, including perturbations of time-one maps of hyperbolic flows~\cite{AVW11};
\item maps with circle center leaves~\cite{AVW12}, including perturbations of certain skew-products, of the type considered in \cite{RW01,SW00}.
\end{itemize}
Moreover, the second alternative is often very rigid: for example, for perturbations of the time-one map of a hyperbolic
flow, it implies that the perturbation is itself the time-one map of a smooth flow.

\subsection{Statement of main results}

Partially hyperbolic diffeomorphisms that are isotopic to Anosov diffeomorphisms have center leaves that are neither compact nor
fixed under the map. The measure-theoretical properties of such center foliations have also been studied by several authors,
especially the intermediate foliations of Anosov diffeomorphisms which we mentioned before.
Saghin, Xia~\cite{SX09} and Gogolev~\cite{Gog12} exhibited conditions under which those intermediate foliations can not
be absolutely continuous. Moreover, Var\~ao~\cite{Var14} gave examples where the disintegration is neither atomic nor
Lebesgue, thus proving that the dichotomy mentioned in the previous paragraph breaks down for such intermediate foliations
of Anosov maps. On the other hand, Ponce, Tahzibi, Var\~ao~\cite{PTV14} prove that atomic disintegration occurs stably in
the isotopy class of certain Anosov automorphisms $A$ of the $3$-torus.

The rigidity phenomenon of \cite{AVW11,AVW12} also does not extend to the non-volume-preserving setting. Indeed, in~\cite{ViY13}
we exhibited stable examples of absolute continuity simultaneously for all invariant foliations (center as well as center-stable
and center-unstable foliations, tangent to $E^{ss}\oplus E^c$ and $E^c\oplus E^{uu}$, respectively).

Our main result in this paper is a criterion for the disintegration of any ergodic measure $\mu$ (not just the volume measure)
to have \emph{uncountable support} along center leaves. By this, we mean that for some choice of a foliation box (in the
sense of \cite[Section~3]{AVW11}) for the center foliation, the supports of the conditional measures of $\mu$ along local
center leaves are uncountable sets.  The criterion applies to partially hyperbolic diffeomorphisms of the $3$-torus in the
isotopy class $\cD(A)$ of an automorphism $A$ as in \eqref{eq.Anosov}:

\begin{main}\label{main.criterion}
Let $\mu$ be an ergodic invariant probability measure of $f\in\cD(A)$ with $h_\mu(f)>\log\lambda_3$.
Then every full $\mu$-measure set $Z \subset M$ intersects almost every center leaf on an uncountable subset.
Moreover, the center Lyapunov exponent along the center direction is non-negative, and even strictly positive if $f$ is $C^2$.
\end{main}

By ``almost every center leaf'' we always mean ``every leaf through every point in some full measure subset''.
By definition, a probability measure has \emph{atomic disintegration} along a foliation if there exists a full measure
set that intersects almost every leaf on a countable subset (see  Appendix~\ref{app.atomic} for a more detailed
discussion of this notion). Thus the first conclusion in the theorem means precisely that the disintegration of $\mu$ along
the center foliation is not atomic.

Let us also point out that the bound $\log\lambda_3$ in Theorem~\ref{main.criterion} is sharp.
Indeed, Ponce, Tahzibi~\cite{PT14} constructed an open set of volume preserving deformations $f$ of a linear Anosov
map $A$ in $\TT^3$ for which the volume measure has atomic disintegration; one can easily find diffeomorphisms in
this open set for which the entropy with respect to the volume measure is equal to $\log\lambda_3$.

As an application of Theorem~\ref{main.criterion}, we obtain stable examples of partially hyperbolic,
volume-preserving diffeomorphisms for which the disintegration of the Lebesgue measure along center leaves is neither
Lebesgue nor atomic:

\begin{maincor}\label{main.example}
Let $f\in\cD(A)$ be a volume-preserving, partially hyperbolic $C^2$ diffeomorphism with $h_{\Leb}(f)>\log\lambda_3$
and whose integrated center Lyapunov exponent is greater than $\log\lambda_2$.
Then there exists a neighborhood $\cU\subset\cD(A)$ of $f$ in the space of volume-preserving $C^2$ diffeomorphisms
such that for every $g\in \cU$ the volume measure is ergodic and its disintegration along the center foliation
(restricted to any foliation box) is neither atomic nor Lebesgue.
\end{maincor}

It follows that for $g\in \cU$ every full volume set intersects almost every center leaf on an uncountable subset.
A related result was obtained by Var\~ao~\cite{Var14}, however, his construction is more
restrictive (it applies only to certain Anosov
diffeomorphisms close to the linear automorphism $A$) and, in particular, it is not known to be stable.

More precise versions of Theorem~\ref{main.criterion} and Corollary~\ref{main.example} will be presented later.
We also provide examples of yet another kind of measure-theoretical behavior of invariant foliations:
for maps of the type constructed by Kan~\cite{Ka94}, we show that the disintegration of Lebesgue along center
leaves may be absolutely continuous but not equivalent to Lebesgue measure. The precise statement is given in
Theorem~\ref{t.kantype}.

\section{Dimension theory for the center foliaton}

In this section, we prove the following theorem:

\begin{theorem}\label{t.entropyestimation}
Let $\mu$ be any ergodic invariant measure of the linear automorphism $A$. If $h_\mu(A) > \log \lambda_3$,
then every full $\mu$-measure subset $Z$ intersects almost every center leaf in an uncountable set.
\end{theorem}

The proof of Theorem~\ref{t.entropyestimation}, which will be given at the end of Section~\ref{ss.mainproposition},
is based on the notion of \emph{partial entropy} of an ergodic probability measure along an expanding foliation,
that we describe in Section~\ref{ss.entropy}.

We prove that if the partial entropy is positive then the invariant measure satisfies the conclusion of the theorem.
The other half of the argument is to prove that the partial entropy is indeed positive under the assumptions
of Theorem~\ref{t.entropyestimation}. This is based on an inequality for partial entropies that is stated in
Proposition~\ref{p.newentropyformula} and which is inspired by results of Ledrappier, Young~\cite{LY85b}.

We also get that if the partial entropy is zero then the foliation constitutes a measurable partition,
in the sense of Rokhlin~\cite{Rok67a}, and the conditional measures are Dirac masses. This seems to be known already,
at least for extreme (strong-unstable) laminations, check Ledrappier, Xie~\cite[Remark~1]{LeX11}.

The construction of Ponce, Tahzibi that we mentioned before also shows that the bound $\log\lambda_3$
is sharp in Theorem~\ref{t.entropyestimation}. Indeed, by a result of Franks~\cite{Fra70} (see also
Section~\ref{s.semiconjugacy}) the diffeomorphisms in the open set constructed in~\cite{PT14} are semiconjugate to $A$.
Using this semiconjugacy, one finds an ergodic invariant measure $\mu$ of $A$ whose disintegration is atomic
and whose entropy is equal to $\log\lambda_3$.


\subsection{Entropy along an expanding foliation}\label{ss.entropy}

Let $f:M\to M$ be a diffeomorphism. We say that a foliation $\cF$ is \emph{expanding} if it is invariant and
the derivative $Df$ restricted to the tangent bundle of $\cF$ is uniformly expanding.
It is a classical fact (check~\cite{HPS70}) that if $f$ admits an invariant dominated splitting
$TM = E^{cs}\oplus E^{uu}$ such that $Df \mid E^{uu}$ is uniformly expanding, then $E^{uu}$ is uniquely
integrable; in this case, the integral foliation $\cF^{uu}$ is an example of expanding foliation.
In general, given an expanding foliation, its tangent bundle may not correspond to the strongest expansion
and an invariant transverse sub-bundle need not exist either.

Let $\cF$ be an expanding foliation, $\mu$ be an invariant probability measure, and $\xi$ be a measurable
partition of $M$ with respect to $\mu$.
We say that $\xi$ is \emph{$\mu$-subordinate} to the foliation $\cF$ if for $\mu$-almost every $x$, we have
\begin{itemize}
\item[(A)] $\xi(x) \subset \cF(x)$ and $\xi(x)$ has uniformly small diameter inside $\cF(x)$;
\item[(B)] $\xi(x)$ contains an open neighborhood of $x$ inside the leaf $\cF(x)$;
\item[(C)] $\xi$ is an \emph{increasing partition}, meaning that $f\xi \prec \xi$.
\end{itemize}

\begin{remark}\label{r.partition}
Ledrappier, Strelcyn~\cite{LS82} proved that the Pesin unstable lamination admits some $\mu$-subordinate
measurable partition. The same is true for the strong-unstable foliation $\cF^{uu}$ of any partially
hyperbolic diffeomorphism. In fact, their construction extends easily to any expanding invariant foliation
$\cF$ (including the center foliations of the maps we consider here), as we are going to sketch
(see also~\cite[Lemma 3.2]{Yang16}).
Start by choosing a finite partition $\cA$ with arbitrarily small diameter such that its elements have
small boundary, in the following sense: there exists $c$ smaller than and close to $1$
\begin{equation}\label{eq.small_boundary}
\sum_{k\ge 1} \sum_{A\in\cA} \mu(B_{c^k}(\partial A)) < \infty.
\end{equation}
Let $\cA^\cF$ be a refinement of $\cA$ whose elements are the intersections of elements of $\cA$
with local plaques of $\cF$. Then the partition
$$
\bigvee_{i=0}^\infty f^i(\cA^\cF)
$$
is $\mu$-subordinate to $\cF$.
\end{remark}

In all that follows, it is assumed that $\mu$-subordinate partitions are constructed in this way.
Indeed, this construction yields the following additional property that will be useful later:
\begin{itemize}
\item[(D)] for any $y\in\cF(x)$ there exists $n\ge 1$ such that $f^{-n}(y) \in \xi(f^{-n}(x))$.
\end{itemize}
Let us explain this, since it is not explicitly stated in the previous papers.
Property (B) ensures that there exists a measurable function $x \mapsto r(x)>0$ such that $\xi(x)$
contains the ball of radius $r(x)$ around $x$ inside the leaf $\cF(x)$. By recurrence,
the pre-orbit $f^{-n}(x)$ of $\mu$-almost point returns infinitely often to any region where
$r(\cdot)$ is bounded from zero. On the other hand, the distance from $f^{-n}(y)$ to $f^{-n}(x)$
goes to zero as $n\to\infty$. Thus property (D) follows.

We also need some terminology from \cite[\S~5]{Rok67a}. Given measurable partitions $\eta_1$ and $\eta_2$, let
$H_\mu(\eta_1 \mid \eta_2)$ denote the mean conditional entropy of $\eta_1$ given $\eta_2$.
The entropy of $f$ with respect to a measurable partition $\eta$ is defined by
$h_\mu(f,\eta)=H_\mu(\eta \mid f\eta^+)$ where $\eta^+=\bigvee_{i=0}^\infty f^i\eta$.
Thus $h_\mu(f,\eta)=H_\mu(\eta \mid f\eta)$ whenever $\eta$ is an increasing measurable partition.

The following result is contained in Lemma~3.1.2 of Ledrappier, Young~\cite{LY85a}:

\begin{lemma}\label{l.definitionleafentropy}
Given any expanding foliation $\cF$, we have $h_\mu(f,\xi_1)=h_\mu(f,\xi_2)$
for any measurable partitions $\xi_1$ and $\xi_2$ that are $\mu$-subordinate to $\cF$.
\end{lemma}

This allows us to define the \emph{partial $\mu$-entropy} $h_\mu(f,\cF)$ of an expanding foliation $\cF$
to be $h_\mu(f,\xi)$ for any $\mu$-subordinate partition. Our next goal is to prove that the nature of
the conditional probabilities of $\mu$ along the leaves of the foliation $\cF$ is directly related to
whether the entropy is zero or strictly positive. That is the content of Propositions~\ref{p.atomic}
and~\ref{p.continuous} below. Beforehand, we must introduce a few important ingredients.

Let $\xi$ be any measurable partition $\mu$-subordinate to $\cF$.
Let $\{\mu_x: x\in M\}$ be the disintegration of $\mu$ with respect to $\xi$.
By definition, $\mu_x(\xi(x))=1$ for $\mu$-almost every $x$.
Keep in mind that $H_\mu(\xi \mid f \xi)=H_\mu(f^{-1}\xi \mid \xi)$, because $\mu$ is $f$-invariant.
Moreover, the definition gives that
\begin{equation}\label{eq.hg}
H_\mu(f^{-1}\xi \mid \xi) = \int g \, d\mu,\quad\text{where } g(x)=-\log \mu_x\big((f^{-1}\xi)(x)\big).
\end{equation}

Let $d_{\cF}(\cdot,\cdot)$ denote the distance along $\cF$-leaves. Given any  $x\in M$, $n\ge 0$ and $\vep>0$, let
$$
B_\cF(x,n,\vep)=\{y\in \cF(x): d_{\cF}(f^i(x),f^i(y))<\vep \text{ for } 0 \leq i < n \}.
$$
Then define
$$
\begin{aligned}
\underline{h}_\mu(x,\vep, \xi)
& = \liminf_{n\to \infty} -\frac{1}{n}\log \mu_x(B_\cF(x,n,\vep))\\
\overline{h}_\mu(x,\vep,\xi)
& = \limsup_{n\to \infty} -\frac{1}{n}\log \mu_x (B_\cF(x,n,\vep)).
\end{aligned}
$$

The following statement is contained in Ledrappier-Young~\cite[\S\S (9.2)-(9.3)]{LY85b}:

\begin{proposition}\label{p.calculation}
At $\mu$-almost every $x$,
$$
\lim_{\vep\to 0} \underline{h}_\mu(x,\vep,\xi)=\lim_{\vep\to 0} \overline{h}_\mu(x,\vep,\xi)=H_\mu(\xi|f\xi).
$$
\end{proposition}

\begin{proof}
The proof that $\lim_{\vep \to 0}\underline{h}_\mu(x,\vep, \xi)\geq H_\mu(\xi\mid f \xi)$ is identical
to \cite[\S~(9.2)]{LY85b} and so we omit it. To prove that
$$
\lim_{\vep \to 0}\overline{h}_\mu(x,\vep, \xi)\leq H_\mu(\xi\mid f \xi),
$$
we could invoke \cite[\S~(9.3)]{LY85b}. However, since we take $\cF$ to be (uniformly) expanding,
it is possible to give a much shorter argument, as follows.

Property (A) above implies that for any $\vep>0$ there is $k_\vep(x) \ge 1$ such that
$$
\diam_{\cF}(f^{-m}\xi(x))<\vep
\quad\text{for any $m \ge k_\vep(x)$.}
$$
This ensures that, for every $x$, $n \ge 1$ and $m \geq k_\vep(x)$,
$$
\big(\bigvee_{j=0}^{n+m}f^{-j}\xi\big)(x) \subset B_\cF(x,n,\vep).
$$
Then
\begin{equation*}
\begin{aligned}
\overline{h}(x,\vep, \xi)
& = \limsup_{n\to \infty} -\frac{1}{n}\log \mu_x(B_\cF(x,n,\vep))\\
& \leq \limsup_{n\to \infty} -\frac{1}{n}\log \mu_x\big((\bigvee_{j=0}^{n+k_\vep(x)}f^{-j}\xi) (x)\big)\\
&=\limsup_{n\to \infty}\frac{1}{n}\sum_{j=0}^{n+k_\vep(x)-1} g(f^j(x)).
\end{aligned}
\end{equation*}
By ergodicity, and the Birkhoff theorem, this means that
\begin{equation*}
\overline{h}(x,\vep, \xi)
\leq \int g \, d\mu = H_\mu(f^{-1}\xi \mid \xi).
\end{equation*}
This proves the claim.
\end{proof}

\subsection{Partial entropy and disintegration}

We are ready to prove that vanishing partial entropy corresponds to an atomic disintegration:

\begin{proposition}\label{p.atomic}
The following conditions are equivalent:
\begin{itemize}
\item[(a)] $h_\mu(f,\cF)=0$;
\item[(b)] there is a full $\mu$-measure subset that intersects each leaf on exactly one point;
\item[(c)] there is a full $\mu$-measure subset that intersects each leaf on a countable subset.
\end{itemize}
\end{proposition}

\begin{proof}
Let $\xi$ be any partition $\mu$-subordinate to $\cF$.

To prove that (c) implies (a), let $\Gamma$ be a full $\mu$-measure subset whose intersection with every
leaf is countable. Replacing $\Gamma$ by a suitable full $\mu$-measure subset, we may assume that
the conclusion of Proposition~\ref{p.calculation} holds, $\mu_x$ is well defined and $\mu_x(\Gamma\cap\xi(x))=1$
for any point $x\in \Gamma$. The latter implies that $\mu_x$ is an atomic measure,
because $\Gamma \cap \xi(x)$ is taken to be countable.
Take any $y\in \Gamma \cap \xi(x)$ such that $\mu_x(\{y\})>0$. Since $\mu_x=\mu_y$, because $\xi(x)=\xi(y)$,
one gets that
$$
\mu_y(B_\cF(y,n,\vep)) \ge \mu_y(\{y\}) > 0
\quad\text{for any $\vep>0$ and $n \geq 1$.}
$$
In view of Proposition~\ref{p.calculation}, this implies that $h_\mu(f,\cF)=H(\xi \mid f\xi)=0$.

It remains to prove that (a) implies (b). By the relation \eqref{eq.hg}, the assumption
$H(f^{-1}\xi \mid \xi) = h_\mu(f,\xi)=0$ implies that $g(x)=0$ for $\mu$-almost every $x$.
In other words, $\mu_x(f^{-1}\xi(x))=1$ for a full $\mu$-measure subset $A_1$ of values of $x$.
Replacing $f$ by $f^n$ and using the relation (Rokhlin~\cite[~\S 7.2]{Rok67a})
$$
H_\mu(f^{-n}\xi|\xi)=nH_\mu(f^{-1}\xi|\xi)
$$
we conclude that for any $n\ge 1$ there exists a full $\mu$-measure set $A_n$ such that
$\mu_x(f^{-n}\xi(x))=1$ for every $x \in A_n$. Now, our assumptions ensure that the diameter of
$f^{-n}\xi(x)$ decreases uniformly to $0$. Thus, for a full $\mu$-measure set $A_\infty=\cap_{n\ge 1} A_n$
of values of $x$, the measure $\mu_x$ is supported on the point $x$ itself: $\mu_x = \delta_x$.
In particular, $A_\infty \cap \xi(x) = \{x\}$ for every $x\in A_\infty$.

Finally, consider the full $\mu$-measure invariant set $A=\cap_{n\ge 0} f^{n} A_\infty$.
Using property (D) above, we get from the previous paragraph that $A \cap \cF(x) = \{x\}$
for every $x\in A$.
\end{proof}

\begin{remark}\label{r.measurablepartition}
It follows from Proposition~\ref{p.atomic} that if $h_\mu(f,\cF)=0$ then the leaves of $\cF$
define a measurable partition of $M$, with respect to $\mu$. Let us also observe that
$h_\mu(f)=0$ implies $h_\mu(f,\cF)=0$ for every expanding foliation $\cF$. Thus, for example,
if $f:M\to M$ is Anosov then its unstable (or stable) leaves form a measurable partition
with respect to any invariant measure with zero entropy.

It is well-known that such measures fill-in a generic subset of the space of all invariant
probability measures $\mu$. One way to see this is to recall that $\mu \mapsto h_\mu(f)$ is
upper semi-continuous (because $f$ is expansive) and every invariant measure is approximated
by measures supported on periodic orbits (by the Anosov closing lemma). These two observations
imply that  $\{\mu: h_\mu(f) < 1/k\}$ is open and dense, for any $k\ge 1$, and the claim
follows immediately.
\end{remark}

\begin{proposition}\label{p.continuous}
Let $\{\mu_x: x\in M\}$ be the disintegration of $\mu$ with respect to any measurable partition
$\xi$ $\mu$-subordinate to $\cF$. The following conditions are equivalent:
\begin{itemize}
\item[(a)] $h_\mu(f,\cF)>0$;
\item[(b)] for $\mu$-almost every point $x$, the measure $\mu_x$ is continuous,
           that is, it has no atoms.
\end{itemize}
Moreover, any of these conditions implies that any full $\mu$-measure subset $Z$ intersects
almost every leaf of $\cF$ on an uncountable set.
\end{proposition}

\begin{proof}
The fact that (b) implies (a) is a direct consequence of Proposition~\ref{p.atomic},
so let us prove that (a) implies (b). By Proposition~\ref{p.calculation}, there is a
full $\mu$-measure subset $A$ of values of $x$ for which the conditional measure $\mu_x$
is well defined and
$$
\liminf_{n\to \infty} -\frac{1}{n}\log \mu_x\big(B_\cF(x,n,\vep))>0.
$$
Clearly, the latter implies that $\mu_x(\{x\})=0$ for $x\in A$.
Since $\mu_y(A)=1$ for $\mu$-almost every $y$ and $\mu_x=\mu_y$ whenever $\xi(x)=\xi(y)$,
this proves that $\mu_y$ is continuous for $\mu$-almost every $y$, as claimed.

Given any full $\mu$-measure subset $Z$, let $Z_1$ be the subset of points $x \in Z$ such
that $\mu_x$ is a continuous measure. Condition (b) ensures that $Z_1$ has full $\mu$-measure.
Then, by the definition of a disintegration, $\mu_x(Z_1)=1$ for every $x$ in some full
$\mu$-measure set $Z_2 \subset Z_1$. Since $\mu_x$ is continuous and $\mu_x(\xi(x))=1$,
this implies that $Z_1 \cap \xi(x)$ is uncountable for every $x\in Z_2$. In particular,
$Z \cap \cF(x)$ is uncountable for every $x \in Z_2$.
\end{proof}

\subsection{Main proposition\label{ss.mainproposition}}

Now we focus on the case when the dynamics is partially hyperbolic. More precisely,
take $f:M\to M$ be a $C^1$ diffeomorphism admitting an invariant decomposition
$TM = E^c \oplus E^{wu} \oplus E^{uu}$ into three continuous sub-bundles such that
\begin{itemize}
\item[(i)] $\dim E^{wu} = \dim E^{uu} = 1$ and
\item[(ii)] both $Df \mid E^{wu}$ and $Df\mid E^{uu}$ are uniform expansions and
\item[(iii)] $E^{wu}$ dominates $E^c$ and $E^{uu}$ dominates $E^{wu}$.
\end{itemize}

It is a classical fact (see Hirsch, Pugh, Shub~\cite{HPS77}) that the sub-bundles $E^{uu}$ and
$E^{u} = E^{wu} \oplus E^{uu}$ are \emph{uniquely integrable}: there exist unique foliations $\cF^{uu}$
and $\cF^{u}$, respectively, whose leaves are $C^1$ and tangent to these sub-bundles at every point.
Property (ii) implies that these foliations are expanding. Moreover, $\cF^{uu}$ sub-foliates $\cF^u$.
We also assume:
\begin{itemize}
\item[(iv)] there exists some invariant \emph{weak-unstable} foliation $\cF^{wu}$ with $C^1$ leaves
that are tangent to $E^{wu}$ at every point.
\end{itemize}
Again, such a foliation is necessarily expanding. Moreover, it sub-foliates $\cF^u$.

We say that $\cF^{wu}$ is \emph{uniformly Lipschitz on leaves of $\cF^u$} if there exists $K>0$ such
that the $\cF^{wu}$-holonomy map between any two segments transverse to $\cF^{wu}$ within distance $1$
from each other inside any leaf of $\cF^u$ is $K$-Lipschitz. The main technical result in this paper is:

\begin{proposition}\label{p.newentropyformula}
Suppose that $\cF^{wu}$ is uniformly Lipschitz on leaves of $\cF^u$. Then,
\begin{equation}\label{eq.maininequality}
h_\mu(f,\cF^{u})-h_\mu(f,\cF^{wu}) \leq \tau^{uu},
\end{equation}
where $\tau^{uu}$ is the largest Lyapunov exponent of $f$ with respect to $\mu$ (corresponding to the
the invariant sub-bundle $E^{uu}$).
\end{proposition}

Ledrappier and Young have a similar statement (\cite[Theorem C']{LY85b}) where the roles of
$\cF^{uu}$ and $\cF^{wu}$ are exchanged and the diffeomorphism is assumed to be $C^2$
($C^{1+\epsilon}$ suffices, by Barreira, Pesin, Schmeling~\cite{BPS99}).
In their setting, the lamination $\cF^{uu}$ is automatically Lipschitz inside $\cF^{u}$.
That is not true, in general, for $\cF^{wu}$.

While we were writing this paper, Fran\c cois Ledrappier pointed out to us that a similar result was
obtained by Jian-Sheng Xie~\cite[equation (2.26)]{Xie14} in the context of linear toral
automorphisms. His result would be sufficient for our purposes, but our methods extend to
non-linear maps, and so they should be useful in more generality.

The arguments that follow are essentially borrowed from \cite{LY85b}.
They can be adapted to yield a version of Proposition~\ref{p.newentropyformula} where the sub-bundle
$E^{wu}$ is assumed to be non-uniformly hyperbolic, and to admit a tangent lamination $\cF^{wu}$
satisfying a Lipschitz condition. We do not state it explicitly because it will not be necessary
for our purposes. The following observation shows that, at least in this non-uniformly hyperbolic
setting, the Lipschitz condition can not be omitted:

\begin{remark}
Shub and Wilkinson~\cite{SW00} dealt with $\C^2$ volume-preserving perturbations of a skew-product map
$$
g \times \id: \TT^2\times S^1 \to \TT^2 \times S^1,
$$
where $g$ is a linear Anosov map on the 2-dimensional torus. The perturbation
$f$ is a partially hyperbolic, volume-preserving diffeomorphism with an invariant circle bundle and
whose center Lyapunov exponent $\tau^c$ is positive. The entropy formula (for partial entropy) gives
that  $h_\mu(f,\cF^{u})$ is equal to the sum $\tau^{uu}+\tau^{wu}$ of the two positive Lyapunov
exponents. On the other hand, Ruelle-Wilkinson~\cite{RW01} showed that every center leaf contains
finitely many $\mu$-generic points. Thus, $h_\mu(f,\cF^{wu})=0$ and so \eqref{eq.maininequality} fails
in this case.
\end{remark}

The proof of Proposition~\ref{p.newentropyformula} is given in Subsection~\ref{ss.newentropyformula}.
It is clear that the weak-unstable foliation of a linear Anosov diffeomorphisms is well defined and
uniformly Lipschitz inside leaves of the unstable foliation $\cF^u$. Thus Theorem~\ref{t.entropyestimation}
is an immediate corollary of Proposition~\ref{p.continuous} and Proposition~\ref{p.newentropyformula}.

\subsection{Auxiliary lemmas}

In the section we quote several lemmas from \cite{LY85b} that will be used in the proof of
Proposition~\ref{p.newentropyformula}.

\begin{lemma}[\cite{LY85b}, Lemma 4.1.3]\label{l.LY}
Let ($X,\nu$) be a Lebesgue space, $\pi: X\to \real^n$ be a measurable map, and
$\{\nu_t : t\in \real^n\}$ be a disintegration of $\nu$ with respect to the partition
$\{\pi^{-1}(t): t\in\real^n\}$. Let $\alpha$ be a countable partition of $X$ with
$H_\nu(\alpha)<\infty$. Define $g(x)=\sum_{A\in \alpha} \chi_A(x)g^A(\pi (x))$,
$$
g_\delta(x)=\sum_{A\in \alpha} \chi_A(x)g^A_\delta(\pi(x))
\text{ and }
g_*(x)=\sum_{A\in \alpha}\chi_A(x)g^A_*(\pi(x)),
$$
where $g^A(t)=\nu_t(A)$ for each $A\in\alpha$ and $t\in\real^n$,
$$
g^A_\delta(x)=\frac{1}{\nu(B_\delta^T(x))}\int_{B_\delta^T(x)} g^A \, d\nu
\text{ and }
g^A_*(x)=\inf_{\delta>0}g^A_{\delta}(x),
$$
where $B_\delta^T(x) = \pi^{-1}(B_\delta(\pi(x)))$. Then $g_\delta\to g$ almost everywhere on $X$ and
$$
\int -\log g_* d\nu \leq H_\nu(\alpha)+\log c +1
$$
where $c=c(n)$ is the constant that comes from Besicovitch covering lemma.
\end{lemma}

\begin{lemma}[\cite{LY85b}, Lemma 4.1.4]\label{l.density}
Let $\omega$ be a finite Borel measure on $\real^n$. Then
$$
\limsup_{\vep\to 0} \frac{\log \omega(B_\vep(x))}{\log\vep}\leq n.
$$
\end{lemma}

\subsection{Proof of Proposition~\ref{p.newentropyformula}}\label{ss.newentropyformula}

We are going to prove that, given any $\beta>0$,
\begin{equation}\label{eq.mainbeta}
\tau^{uu}+\beta\geq (1-\beta)[h_\mu(f,\cF^{u})-h_\mu(f,\cF^{wu})-2\beta].
\end{equation}
Proposition~\ref{p.newentropyformula} follows by making $\beta$ go to zero.
Let $\beta>0$ be fixed from now on.
The first step is to construct two suitable $\mu$-subordinate partitions, $\xi^{u}$
and $\xi^{wu}$, for foliations $\cF^{u}$ and $\cF^{wu}$, respectively.

Let $\cA$ be a finite partition with arbitrarily small diameter and whose
elements have small boundary in the sense of \eqref{eq.small_boundary}.
Denote by $\cA^u$ and $\cA^{wu}$ the refinements of $\cA$ defined by
$$
\cA^u(x) = \cF^u_{loc}(x) \cap \cA(x)
\quand
\cA^{wu}(x) = \cF^{wu}_{loc}(x) \cap \cA(x).
$$
Arguing as in Remark~\ref{r.partition}, we see that
$$
\xi^u=\bigvee_{n\geq 0} f^n \cA^u \quand \xi^{wu}=\bigvee_{n\geq 0} f^n(\cA^{wu})
$$
are measurable partitions $\mu$-subordinate to $\cF^u$ and $\cF^{wu}$ respectively.

The next lemma states that $\xi^{wu}$ refines $\xi^{u}$ and the quotient $\xi^{u}/\xi^{wu}$
is preserved by the dynamics:

\begin{lemma}\label{l.productstructure}
Take the diameter of $\cA$ to be sufficiently small.
Then for $\mu$-almost every $x, y \in M$ with $y\in \xi^u(x)$,
\begin{itemize}
\item[(a)] $\xi^u(x)\cap \cF^{wu}_{loc}(y)=\xi^{wu}(y)$ and
\item[(b)] $f(\xi^{wu}(f^{-1}(y)))\cap \xi^u(x)=\xi^{wu}(y)$.
\end{itemize}
\end{lemma}

\begin{proof}
The relation $\supset$ in (a) is clear from the definitions. To prove the converse,
let $y, z\in \xi^u(x)$ be such that $z\in \cF^{wu}_{loc}(y)$.
By the definition of $\xi^u$, the backward iterates $f^{-n}(y)$ and $f^{-n}(z)$
belong to the same element of $\cA^{u}$ and consequently to the same element of $\cA$.
By property (D) applied to the partition $\xi^{wu}$, we have that
$f^{-n}(z)\in\xi^{wu}(f^{-n}(y))$ for every large $n$.
In particular, $f^{-n}(y)$ and $f^{-n}(z)$ belong to the same element of $\cA^{wu}$
for every large $n$. Choose any such $n$. Since $\cA$ is assumed to have small diameter,
$\cA^{wu}(y_{-n})=\cA^{wu}(z_{-n})$ also has small diameter inside the corresponding
$\cF^{wu}$-leaf. Then, by continuity, $f(\cA^{wu}(f^{-n}(y)))$ is contained in
$\cF^{wu}_{loc}(f^{-n+1}(y))$. This proves that
$$
f^{-n+1}(z) \in \cF^{wu}_{loc}(f^{-n+1}(y)) \cap \cA(f^{-n+1}(y))=\cA^{wu}(f^{-n+1}(y)).
$$
By (backwards) induction, this proves that $f^{-n}(y)$ and $f^{-n}(z)$
belong to the same element of $\cA^{wu}$ for every $n$. Thus $\xi^{wu}(y)=\xi^{wu}(z)$,
as we wanted to prove. The proof of part (a) is complete.

From $\xi^{wu}(f^{-1}(y))\subset \cF^{wu}_{loc}(f^{-1}(y))$ we immediately get that
$f(\xi^{wu}(f^{-1}(y)))\subset \cF^{wu}_{loc}(y)$. Combining this with part (a),
we find that
$$
f(\xi^{wu}(f^{-1}(y)))\cap \xi^u(x)\subset \xi^{wu}(y).
$$
This proves the relation $\subset$ in part (b) of the lemma.
To prove the converse, observe that $\xi^{wu}(y) \subset \xi^{u}(x)$, by definition, and
$f(\xi^{wu}(f^{-1}(y)))\supset \xi^{wu}(y)$ because the partition $\xi^{wu}$ is increasing.
\end{proof}

It follows that one may identify each quotient $\xi^{u}(x)/\xi^{wu}$ with a subset of the
local strong-unstable leaf $\cF^{uu}(x)$. Indeed, define
$$
\pi^{wu}_x:\xi^u(x) \to \cF^{uu}_{loc}(x),
\quad
\pi^{wu}_x(y) = \text{ the unique point in } \cF^{wu}_{loc}(y) \cap \cF^{uu}_{loc}(x).
$$
It is clear that this map is constant on every element of $\xi^{wu}$, and part (a) of
Lemma~\ref{l.productstructure} ensures that it is injective.
Thus it induces an injective map from $\xi^{u}(x)/\xi^{wu}$ to $\cF^{uu}_{loc}(x)$.
Using this latter map, we may transport the Riemannian distance on $\cF^{uu}_{loc}(x)$
to a distance $d_x$ on the quotient space $\xi^{u}(x)/\xi^{wu}$.

In what follows, we define yet another distance on $\xi^{u}(x)/\xi^{wu}$ which we are going
to denote as $\tilde{d}$ and which has the advantage of being independent of $x$.
For this, we need a kind of Pesin block construction, which is contained in the next proposition.

\begin{proposition}\label{p.Pliss}
For any $\vep>0$, there is a positive measure subset $\Lambda_\vep$ such that, for any
$x\in \Lambda_\vep$ and any $n>0$,
$$
\frac{1}{n}\log\|Df^n \mid {E^{wu}_x}\|\leq \tau^{wu}+\vep.
$$
\end{proposition}

The arguments are very classical, except for the fact that here the diffeomorphism is only
assumed to be $C^1$, so we just outline the proof of the proposition.
A similar construction appeared in~\cite{Yang07}. Define
$$
\Lambda_\vep=\{x: \frac{1}{n}\log\|Df^n \mid {E^{wu}_x}\| \leq \tau^{wu}+\vep \text{ for every $n > 0$} \}.
$$
Then $\Lambda_\vep$ is a compact set, possibly empty. To prove that $\mu(\Lambda_\vep)>0$ it suffices
to show that the forward orbit $\Orb^+(x)$ of $\mu$-almost every $x$ intersects $\Lambda_\vep$.

By the theorem of Oseledets, for $\mu$-almost every $x$ there exists $n(x)\geq 1$ such that
$$
\frac{1}{n}\log\|Df^n \mid {E^{wu}_x}\|\leq \tau^{wu}+\frac{\vep}{2}
\text{ for every $n\geq n(x)$.}
$$
We also need the following variation of the Pliss lemma (see \cite[Lemma~11.5]{Beyond}):

\begin{lemma}\label{l.Pliss}
Given $K>0$ and $\tau < \bar\tau$ and any sequence $\{a_n\}_{n=1}^\infty$ such that
$\|a_n\|<K$ for every $n\ge 1$ and
$$
\limsup_{n\to \infty} \frac{1}{n} \sum_{j=1}^n a_j <\tau,
$$
there exists $n_0>0$ such that
$$
\frac{1}{m}\sum_{j=1}^m a_{n_0+j}<\bar\tau \text{ for any } m\in \mathbb{N}.
$$
\end{lemma}

Take $K=\sup_{x\in M}\{\|Df(x)\|\}$, $\tau=\tau^{wu}+\vep/2$ and $\bar\tau=\tau^{wu}+\vep$,
and define $a_n=\|Df\mid {E^{wu}_{f^n(x)}}\|$ for $n\ge 1$. From Lemma~\ref{l.Pliss} we get that there is $n(x)>0$ such that
$$
\frac{1}{m} \log \|Df^m \mid {E^{wu}_{f^{n(x)}(x)}}\| \leq \tau^{wu}+\vep \text{ for any } m\geq 1.
$$
Thus $f^{n(x)}(x)\in \Lambda_\vep$, which implies the claim that $\Orb^+(x)$ intersects $\Lambda_\vep$.
This completes our outline of the proof of Proposition~\ref{p.Pliss}.

From now on, let $\Lambda=\Lambda_{\beta/3}$. Fix $r_0>0$ such that for any $x,y\in M$
\begin{equation}\label{eq.rzero}
d(x,y)<r_0 \text{ implies } \|\log Df\mid {E^{wu}_{x}}-\log Df\mid {E^{wu}_{y}}\|\leq \beta/3.
\end{equation}
Assume that the diameter of $\cA$ is smaller than $r_0$. Then the same is true for $\cA^u$ and $\cA^{wu}$,
and so $\xi^u$ and $\xi^{wu}$ also have diameter less than $r_0$.

Fix $x_0 \in \supp(\mu \mid \Lambda)$, that is, such that every neighborhood of $x$ intersects $\Lambda$
on a positive measure subset. Let $D \ni x$ be a small codimension-$1$ disk transverse to $\cF^{wu}$.
Consider local smooth coordinates $(x_1, x_2, \dots, x_{d-1})$ be local coordinates on $D$ such that the
$x_1$-axis is close to the direction of $E^{uu}$. Let $B$ be the union of the local $\cF^{wu}$-leaves
through points of $D$ and $\tilde\pi$, from $B$ to the $x_1$-axis to be the composition of the projection
$B\to D$ along $\cF^{wu}$-leaves with the projection to the $x_1$-axis associated with the chosen coordinates.

\begin{remark}\label{r.uniformLipschtez}
The projections along the local coordinates are smooth maps, of course. Recall that $\cF^{u}$ is
$2$-dimensional and is sub-foliated by $\cF^{uu}$ and by $\cF^{uu}$. Since we assume that the
weak-unstable foliation $\cF^{wu}$ is uniformly Lipschitz inside each leaf of $\cF^{u}$,
we get that $\tilde\pi$ is uniformly bi-Lipschitz restricted to each leaf of $\cF^{uu}$ inside $B$.
Let $\tilde{K}$ be a uniform Lipschitz constant.
\end{remark}

It is no restriction to suppose that $B_{2r_0}(x_0)\subset B$. Define $\Lambda_0=\Lambda \cap B_{r_0}(x_0)$.
By further reducing $r_0>0$ if necessary, we may assume that
\begin{equation}\label{eq.smallmeasure}
e^{-\beta(\tau^{wu}+\beta)} \tilde{K}^{4\mu(\Lambda_0)}<1.
\end{equation}

We can extend the projection $\tilde{\pi}$ from the domain $B$ to the union of $\xi^u(x)$ over all
$x\in \cup_{n\geq 0} f^n(\Lambda_0)$, as follows. Given such an $x$, let $n$ be the smallest nonnegative
integer such that $f^{-n}(x)\in \Lambda_0$. Since $\xi^u$ is increasing and has small diameter,
$f^{-n}(y)\in \xi^u(f^{-n}(x))\subset B_{r_0}(\Lambda_0)\subset B$ for any $y\in \xi^u(x)$.
Just define
$$
\tilde{\pi}(y)=\tilde{\pi}(f^{-n}(y)).
$$
Keep in mind that $\cup_{n\geq 0} f^n(\Lambda_0)$ has full $\mu$-measure, by ergodicity.
Now we are ready to introduce the announced transverse distance $\tilde{d}$:
for $x\in \cup_{n\geq 0} f^n(\Lambda_0)$, and $y_1, y_2\in \xi^u(x)$, define
\begin{equation}\label{eq.dtilde}
\tilde{d}(y_1,y_2)=|\tilde{\pi}(y_1) - \tilde{\pi}(y_2)|.
\end{equation}
By Lemma~\ref{l.productstructure}(a), this function $\tilde{d}(\cdot, \cdot)$ induces a distance
on the quotient space $\xi^u(x)/\xi^{wu}$ which is independent of $x$.

Let $\{\mu^{u}_x: x \in M\}$ and $\{\mu^{wu}_x: x \in M\}$ be the disintegrations of $\mu$ with respect to the partitions
$\xi^u$ and $\xi^{wu}$, respectively. For $\mu$-almost every $x$, consider the disk
$B^T_\rho(x)=\{y\in \xi^u(x): \tilde{d}(x,y)<\rho\}$. We are going to prove that
\begin{equation}\label{eq.dimensioninequality}
(\tau^{wu}+\beta)\limsup_{\rho\to 0} \frac{\log \mu^{u}_x (B^T_\rho (x))}{\log \rho}
\geq (1-\beta) [h_\mu(f,\xi^u)-h_\mu(f,\xi^{wu})-2\beta].
\end{equation}

Our definitions are such that $\mu^{u}_x (B^T_\rho (x))$ coincides with the (projection) measure
of an Euclidean ball of radius $\rho$ in the $x_1$-axis. Since the latter is $1$-dimensional,
the $\limsup$ on the left-hand side is smaller than or equal to $1$ (compare Lemma~\ref{l.density}).
Recalling also the definition of partial entropy, we immediately conclude that
\eqref{eq.dimensioninequality} implies \eqref{eq.mainbeta}.
Thus we have reduced the proof of Proposition~\ref{p.newentropyformula} to proving \eqref{eq.dimensioninequality}.

The rest of the argument is based on Lemma~\ref{l.LY}. Define $g, g_*, g_\delta: M\to \real$ by
$$
\begin{aligned}
g(y)=\mu_y^{wu}((f^{-1}\xi^{u})(y)) \text{ and } & g_*(y)=\inf_{\delta>0} g_\delta(y) \text{ with}\\
g_\delta(y)=\frac{1}{\mu^{u}_y(B^T_\delta(y))}\int_{B^T_\delta(y)} g(y) \, d\mu^u_y(dz)
&=\frac{\mu^u_y((f^{-1}\xi^u)(y)\cap B^T_\delta(y))}{\mu^{u}_y(B^T_\delta(y))}
\end{aligned}
$$
(the last identity is a consequence of the definition of disintegration). It follows from Lemma~\ref{l.LY}
that
\begin{equation}\label{eq.deltagdelta}
g_\delta\to g \text{ at $\mu$-almost everywhere and } \int-\log g_* d\mu <\infty.
\end{equation}
To see this, just fix $x$, substitute $(\xi^{u}(x),\mu^u_x)$ for $(X,\nu)$, let $\tilde{\pi}$
be the projection from $\xi^u(x)$ to the $x_1$-axis introduced previously,
and take $\alpha=f^{-1}\xi^u| \xi^u(x)$; finally, integrate with respect to $\mu$.

By Poincar\'e recurrence, for $\mu$-almost every $x\in \Lambda_0$ one can find times
$0=n_0 < n_1 < \dots < n_j < \cdots <n$ such that $f^{n_j}(x) \in \Lambda_0$ for any $j \geq 0$.
For each $0 \leq k < n$, take $j\ge 0$ largest such that $n_j \le k$ and then define
$$
a(x,n,k)=B^T_{\delta(x,n,k)}(f^{k}x)
\text{ with }
\delta(x,n,k)=e^{-(n-n_j)(\tau^{wu}+\beta)}\tilde{K}^{2j}.
$$
Note that $\delta(x,n,k)=\delta(x,n,n_j)$ for every $k\in\{n_j, \dots, n_{j+1}-1\}$.
This will be used for proving the following invariance property:

\begin{lemma}\label{l.transversecontracting}
$a(x,n,k)\cap (f^{-1}\xi^u)(f^k (x))\subset f^{-1}(a(x,n,k+1))$ for every $x\in \Lambda_0$.
\end{lemma}

\begin{proof}
Suppose first that $k \neq n_{j+1}-1$. Note that $a(x,n,k)$ consists of elements $\xi^{wu}(y)$ of
the weak-unstable partition and, of course, the same is true for $a(x,n,k+1)$.
For each one of them, Lemma~\ref{l.productstructure}(b) ensures that
$$
f(\xi^{wu}(y)) \cap \xi^{u}(f^{k+1}(x)) = \xi^{wu}(f(y))
$$
for any $y\in a(x,n,k)\cap (f^{-1}\xi^u)(f^k (x))$.
The definition \eqref{eq.dtilde} ensures that $\tilde{d}(y,f^k(x)) = \tilde{d}(f(y),f^{k+1}(x))$.
Besides, as observed before, the transverse diameters $\delta(x,n,k)$ and $\delta(x,n,k+1)$
are the same in the present case. In this way we get that
$$
f(a(x,n,k))\cap \xi^{u}(f^{k+1}(x))=a(x,n,k+1),
$$
as we wanted to prove.

From now on, suppose that $k=n_{j+1}-1$. While the transverse diameters are no longer necessary the
same for $k$ and $k+1$, all we have to do is that it is still true that
\begin{equation}\label{eq.anotherdtilde}
\tilde{d}(f(y),f^{k+1}(x)) \le \delta(x,n,k+1)
\end{equation}
for any $y\in a(x,n,k)\cap (f^{-1}\xi^u)(f^k (x))$. By definition,
$$
\tilde{d}(y,f^{k}(x))\leq e^{-(n-n_{j})(\tau^{wu}+\beta)} \tilde{K}^{2j} .
$$
According to Remark~\ref{r.uniformLipschtez}, this implies that
$$
d_{f^{n_{j}}(x)}(f^{n_{j}-n_{j+1}+1}(y),f^{n_{j}}(x))
\leq \tilde{K}^{2j+1} e^{-(n-n_{j})(\tau^{wu}+\beta)}.
$$
Since $f^{n_{j}}(x) \in \Lambda_0 \subset \Lambda_{\beta/3}$ and $\diam(\xi^u)<r_0$,
Proposition~\ref{p.Pliss} together with our choice of $r_0$ ensure that
$$
d_{f^{n_{j+1}}(x)}(f(y),f^{n_{j+1}}(x))
\leq \tilde{K}^{2j+1} e^{-(n-n_{j+1})(\tau^{wu}+\beta)}.
$$
Using Remark~\ref{r.uniformLipschtez} once again, it follows that
$$
\tilde{d}(f(y),f^{n_{j+1}}(x))
\leq \tilde{K}^{2j+2} e^{-(n-n_{j+1}) (\tau^{wu}+\beta)}.
$$
This means that $f(y)\in a(x,n,k+1)$, as we wanted to prove.
\end{proof}

Now let us estimate the measure $\mu^{u}_x(a(x,n,0))$ for $x\in \Lambda_0$. Clearly,
\begin{equation}\label{eq.xnzero}
\frac{\mu^u_x (a(x,n,0))}{\mu^u_{f^{n}(x)}(a(x,n,p(n)))} =  \prod_{k=0}^{p(n)-1} \frac{\mu^u_{f^{k}(x)}(a(x,n,k))}{\mu^u_{f^{(k+1)}x}(a(x,n,k+1))},
\end{equation}
where $p(n)=[n(1-\beta)]$. For each $0\leq k \leq p(n)-1$ and $\mu$-almost every $x\in\Lambda_0$,
$$
\frac{\mu^u_{f^{k}(x)}(a(x,n,k))}{\mu^u_{f^{(k+1)}(x)}(a(x,n,k+1))}
= \mu^u_{f^{k}(x)}(a(x,n,k))\frac{\mu^u_{f^{k}(x)}(f^{-1}(\xi^u(f^{(k+1)}(x))))} {\mu^u_{f^{k}(x)}(f^{-1}(a(x,n,k+1)))}
$$
(use the essential uniqueness of the disintegration together with the fact that $\xi^u$ is an
increasing partition). By Lemma~\ref{l.transversecontracting}, the right-hand side is bounded above by
$$
\frac{\mu^u_{f^{k}(x)}(a(x,n,k))}{\mu^u_{f^{k}(x)}(f^{-1}\xi^u(f^{k}(x))\cap a(x,n,k))}\mu^u_{f^{k}(x)}(f^{-1}\xi^u(f^{k}(x))).
$$

The quotient on the left-hand side is precisely $1/g_{\delta(x,n,k)}(f^k(x))$.
Write the last factor as $e^{-I(f^{k}(x))}$, where $I(z)=-\log \mu^u_z(f^{-1}\xi^u(z))$.
Replacing these expressions in \eqref{eq.xnzero}, we get that
$$
\begin{aligned}
\log \mu^u_x (B^T_{e^{-n(\tau^{wu}+\beta)}}(x))
& = \log \mu^u_x (a(x,n,0))
\le \log \frac{\mu^u_x (a(x,n,0))}{\mu^u_{f^{n}(x)}(a(x,n,p))} \\
& \leq -\sum_{k=0}^{p(n)-1}\log g_{\delta(x,n,k)}(f^{k}(x))-\sum_{k=0}^{p(n)-1}I(f^{k}(x)).
\end{aligned}
$$
Consequently,
$$
\begin{aligned}
(\tau^{wu}+\beta)\limsup_{\rho\to 0} & \frac{\log \mu^u_x (B^T(x,\rho))}{\log\rho}
\ge (\tau^{wu}+\beta) \liminf_{n\to \infty} \frac{\log \mu^u_x (B^T_{e^{-n(\tau^{wu}+\beta)}}(x))}{\log e^{-n(\tau^{wu}+\beta)}}\\
& \ge \liminf_{n\to\infty} \big[\frac{1}{n}\sum_{k=0}^{p(n)-1}\log g_{\delta(x,n,k)}(f^{k}(x))+\frac{1}{n}\sum_{i=0}^{p(n)-1}I(f^{k}(x))\big].
\end{aligned}
$$
By the Birkhoff ergodic theorem and the definition of conditional entropy
$$
\lim_{p \to\infty} \frac{1}{p}\sum_{i=0}^{p}I(f^{k}(x)))
= \int I \, d\mu = H_\mu(f^{-1}\xi^u \mid \xi^u).
$$
Therefore, using also the definition of partial entropy of an expanding foliation,
\begin{equation}\label{eq.2ndlimit}
\lim_{n \to\infty} \frac{1}{n}\sum_{i=0}^{p(n)}I(f^{k}(x)))
= (1-\beta) h_\mu(f,\cF^u).
\end{equation}

We are left to prove that
\begin{equation}\label{eq.1stlimit}
\limsup_{n\to\infty} -\frac{1}{n} \sum_{k=0}^{p(n)} \log g_{\delta(x,n,k)}(f^{k}x)\leq (1-\beta)(h_\mu(f,\cF^{wu})+2\beta).
\end{equation}
By \eqref{eq.deltagdelta}, we may find a measurable function $x\mapsto \delta(x)$ such that
$$
-\log g_{\delta}(x)\leq -\log g(x)+\beta
\text{ for every } \delta<\delta(x)
$$
and a constant $\delta_0>0$ such that
$$
\int_{\{x: \delta(x) \le \delta_0\}} -\log g_*(x)\, d\mu(x)<\beta.
$$

By ergodicity, for $\mu$-almost all $x$ we have $\#\{0 \leq i < n: f^i(x) \in \Lambda_0\} \leq 2n \mu(\Lambda_0)$
for every large $n$. In particular, we always have $j \le 2n\mu(\Lambda_0)$.
Moreover, $n_j \le k \le p(n)$ implies that $n-n_j \ge \beta n$. Therefore,
$$
\delta(x,n,k)
= e^{-(n-n_j)(\tau^{wu}+\beta)} \tilde{K}^{2j}
\leq e^{-\beta n (\tau^{wu}+\beta)} \tilde{K}^{4n\mu(\Lambda_0)}
$$
for every $0\le k \le p(n)$. By \eqref{eq.smallmeasure}, this implies that $\delta(x,n,k) \to 0$ uniformly in
$0 \le k \le p(n)$ when $n\to\infty$. In particular, $\delta(x,n,k) < \delta_0$ for every $k\leq p(n)$ if $n$
is sufficiently large.

Going back to \eqref{eq.1stlimit}, note that
$$
\begin{aligned}
\sum_{k=0}^{p(n)} -\log g_{\delta(x,n,k)}& (f^{k}x)\\
& \leq \sum_{k: \delta(f^k(x)) > \delta_0} -\log g(f^k(x))+\beta +\sum_{k: \delta(f^k(x)) \le \delta_0} -\log g_*(f^k(x))
\end{aligned}
$$
and, by the Birkhoff ergodic theorem, this leads to
$$
\begin{aligned}
\limsup_n -\frac{1}{n} \sum_{k=0}^{p(n)} \log g_{\delta(x,n,k)}& (f^{k}x)\\
& \leq (1-\beta) \big[ \int-\log g \, d\mu + \beta +\int_{\{x: \delta(x) \le \delta_0\}}-\log g_*d\mu\big]\\
& \leq (1-\beta) \big[ \int-\log g \, d\mu + 2\beta\big].
\end{aligned}
$$
To conclude, note that $g(x)=\mu^{wu}_x (f^{-1}\xi^{wu}(x))$ and so
$$
\int -\log g \, d\mu = H_\mu(f^{-1}\xi^{wu} \mid \xi^{wu}) = h_\mu(f,\cF^{wu}).
$$
This completes the proof of \eqref{eq.dimensioninequality} and thus of Proposition~\ref{p.newentropyformula}.

\section{Semiconjugacy to the linear model}\label{s.semiconjugacy}

Let $f:\TT^3\to\TT^3$ be a $C^1$ diffeomorphism in the isotopy class $\cD(A)$ of a linear automorphism
$A$ as described in the previous section. By Potrie~\cite{Pot15}, such a diffeomorphism is \emph{dynamically
coherent}: there exist invariant foliations $\cF^{cs}$ and $\cF^{cu}$ tangent to the center-stable and
center-unstable sub-bundles, respectively. Intersecting their leaves, one obtains an invariant center
foliation tangent to $E^c$.

By Franks~\cite{Fra70}, there exists a continuous surjective map $\phi:\TT^3\to\TT^3$ that
semiconjugates $f$ to $A$, that is, such that $\phi \circ f = A \circ \phi$. Moreover, by construction,
$\phi$ lifts to a map $\tphi:\RR^3\to\RR^3$ that is at bounded distance from the identity:
there exists $C>0$ such that
\begin{equation}\label{eq.closetoidentity}
\|\tphi(\tilde x)-\tilde x\| \le C
\quad\text{for every $\tilde x\in\RR^3$.}
\end{equation}

\subsection{Geometry of the center foliation}

\begin{proposition}\label{p.semiconjugation}
For all $z\in \TT^3$, the pre-image $\phi^{-1}(z)$ is a compact connected subset of some center leaf of $f$
(that is, either a point or an arc) with uniformly bounded length.
\end{proposition}

This was proven by Ures~\cite[Proposition 3.1]{Ure12}, assuming absolute partial hyperbolicity,
and by Fisher, Potrie, Sambarino~\cite{FPS14}, in the present context. We outline the arguments,
to highlight where the uniform bound comes from.

\begin{proof}[Sketch of proof of Proposition~\ref{p.semiconjugation}:]
Let $\tilde f$ and $\tilde A$ be the lifts of $f$ and $A$, respectively, to the universal cover $\RR^3$
of $\TT^3$. The center foliations $\cF^c_f$ and $\cF^c_A$ also lift to foliations $\tilde\cF^c_f$ and
$\tilde\cF^c_A$ in $\RR^3$ and these are center foliations for $\tilde f$ and $\tilde A$, respectively.
We need the following facts:

\begin{enumerate}
\item[(i)] $\tphi(\tilde x) = \tphi(\tilde y)$ if and only if there exists $K>0$ such that
$\|\tilde f^n(\tilde x) - \tilde f^n(\tilde y)\| < K$ for all $n\in \ZZ$. In fact, $K$
may be chosen independent of $\tilde x$ and $\tilde y$.

\item[(ii)] There exists a homeomorphism $h:\TT^3\to\TT^3$ which maps each center leaf $L$
of $f$ to a center leaf of $A$ so that $h(f(L)) = A(h(L))$ for every $L$.  We say that $h$ is
a leaf conjugacy between $f$ and $A$.

\item[(iii)] The leaves of $\tilde\cF^c_f$ are quasi-isometric in $\RR^3$: there exists $Q>1$ such
that  $d_c(x,y) \le Q \|x-y\| + Q$ for every $x, y$ in the same center leaf, where $d_c$
denotes the distance inside the leaf.
\end{enumerate}

Fact (i) is a  direct consequence of the construction of $\phi$ in Franks~\cite{Fra70},
which is based on the shadowing lemma for the linear automorphism $A$.
Fact (ii) was proven in Hammerlindl, Potrie~\cite[Corollary~1.5]{HaP14}.
See Hammerlindl~\cite[Proposition~2.16]{Ham13} for a proof of fact (iii) in the absolute
partially hyperbolic case and  Hammerlindl, Potrie~\cite[Section~3]{HaP14} for an explanation
on how to extend the conclusion to the present context.

The map $h$ in (ii) lifts to a homeomorphism $\tilde h:\RR^3\to\RR^3$ which is a leaf conjugacy
between $\tilde f$ and $\tilde A$. From a general property of lift maps, we get that
$$
d(\tilde{h}(\tilde{x}_n),\tilde{h}(\tilde{y}_n))\to \infty \quad\Rightarrow\quad  d(\tilde{x}_n,\tilde{y}_n)\to \infty.
$$
It is clear that given any distinct $\tilde\cF^c_A$-leaves $F_1$ and $F_2$ the
distance between $\tilde A^n(F_1)$ and $\tilde A^n(F_2)$ goes to infinity when $n\to+\infty$
or $n\to-\infty$ or both. In view of the previous observation, the same is true for $\tilde f$:
given any distinct $\tilde\cF^c_f$-leaves $L_1$ and $L_2$ the distance between $\tilde f^n(L_1)$
and $\tilde f^n(L_2)$ goes to infinity when $n\to+\infty$ or $n\to-\infty$ or both. So,
by fact (ii) above, every pre-image $\phi^{-1}(z)$ is contained in a single $\tilde\cF^c_f$-leaf.

On the other hand, the quasi-isometry property (iii) implies that if two points $\tilde x$ and
$\tilde y$ are such that $\|\tilde f^n(\tilde x) - \tilde f^n(\tilde y)\|$, $n\in \ZZ$ is bounded then,
denoting by $[\tilde x, \tilde y]_c$ the center segment between the two points, the length of
$\tilde f^n([\tilde x, \tilde y]_c)$, $n\in\ZZ$ is also bounded. That implies that the whole
segment is contained in the same pre-image $\phi^{-1}(z)$.
\end{proof}

\begin{proposition}\label{p.center-to-center}
The image under $\phi:\TT^3\to\TT^3$ of any center leaf of $f$ is contained in some center leaf of $A$.
\end{proposition}

\begin{proof}
We claim that the image of any center-stable leaf of $f$ is contained in a leaf of the center-stable
foliation of $A$; note that the center-stable leaves of $A$ are just the translates of the
center-stable subspace $E^{cs}$. Analogously, one gets that the image of any center-unstable leaf
of $f$ is contained in a translate of the center-unstable subspace $E^{cu}$. Taking intersections,
we get that the image of every center leaf of $f$ is contained in a center leaf of $A$, as stated.
So, we have reduced the proof of the proposition to proving this claim.

Let $\cF_f^{cs}$ be the center-stable foliation of $f$ and $\tilde\cF^{cs}_f$ be its leaf to the
universal cover $\RR^3$. By Potrie~\cite{Pot15} (Theorem~5.3 and Proposition~A.2), there exists
$R>0$ such that every leaf of $\tilde\cF^{cs}_f$ is contained in the $R$-neighborhood of some
translate $\tilde x+E^{cs}$ of the plane $E^{cs}$ inside $\RR^3$.
Then, since $\|\tphi-\id\|\le C$, the image of every leaf of $\tilde\cF^{cs}_f$ under
$\tphi$ is contained in the $R+C$-neighborhood of $\tilde x+E^{cs}$.

Since $\tilde\cF^{cs}_f$ is invariant under $\tilde f$ and $\tphi$ semi-conjugates $\tilde f$
to $\tilde A$, the family of images $\tphi(\tilde\cF^{cs}_f(\tilde y))$, $\tilde y\in\RR^3$
is invariant under $\tilde A$. The family of translates of $\tilde y + E^{cs}$, $\tilde y\in\RR^3$
is also invariant under $\tilde A$, of course. Moreover, $\tilde A$ is expanding in the direction
transverse to $E^{cs}$. Thus, the only way the conclusion of the previous paragraph may occur is
if every image $\tphi(\tilde\cF^{cs}_f(\tilde y))$ is contained in $\tphi(\tilde y)+E^{cs}$.

Projecting this conclusion down to the torus, we get our claim.
\end{proof}

\begin{corollary}\label{c.leaf_unique}
The pre-image under $\phi:\TT^3\to\TT^3$ of any center leaf of $A$ consists of a unique center leaf of $f$.
\end{corollary}

\begin{proof}
Let$\cF^c_A(y)$ be any center-leaf of $A$. Proposition~\ref{p.center-to-center} implies that its pre-image
is a union of center leaves and Proposition~\ref{p.semiconjugation} implies that the images of
these center leaves are pairwise disjoint. So, by connectivity, it suffices to prove that if the
image of a center leaf $\cF^c_{f}(x)$ is contained in $\cF^c_A(y)$ then this image is an open subset
of $\cF^c_A(y)$. For that, it suffices to observe that the map $\phi:\cF^c_{f}(x)\to\cF^c_A(y)$ is
monotone, which is an immediate consequence of Proposition~\ref{p.semiconjugation}.
\end{proof}

\begin{corollary}\label{c.measurability}
$Y=\{y\in\TT^3: \phi^{-1}(y) \text{ consists of a single point}\}$ is a Borel set and the restriction $\phi:\phi^{-1}(Y)\to Y$ is a
homeomorphism with respect to the relative topologies. In particular, the inverse $\phi^{-1}$ is a measurable map.
\end{corollary}

\begin{proof}
We already know that $\phi^{-1}(y)$ is always a center leaf segment. Since $\phi$ is continuous, the length of
this segment is an upper semi-continuous (hence measurable) function of $y$. In particular, the set $Y$ of
points where the length is equal to zero is measurable, as claimed in the first part of the corollary.

It is clear that the restriction $\phi:\phi^{-1}(Y)\to Y$ is a continuous bijection. So, to prove the second part
of the corollary it suffices to check that it is also a closed map. By definition, every closed subset of
$\phi^{-1}(Y)$ may be written as $K \cap \phi^{-1}(Y)$ for some compact subset of $\TT^3$. Observe that
$\phi\big(K \cap \phi^{-1}(Y)\big) = \phi(K) \cap Y$ and this is a closed subset of $Y$, because $\phi(K)$ is a
compact subset of $\TT^3$. This proves that $f$ is indeed a closed map.
\end{proof}

\subsection{Diffeomorphisms derived from Anosov}

Let $\phi_*$ denote the push-forward map induced by $\phi$ in the space of probability measures.
We have the following general result:

\begin{proposition}\label{p.surjectivity}
Consider continuous maps $g:M\to M$ and $h:N\to N$ in compact spaces and suppose there exists a
continuous surjective map $p:M\to N$ such that $p\circ g = h \circ p$. Then:

\begin{itemize}
\item[(a)] The push-forward $p_*$ maps the space of $g$-invariant (respectively, $g$-ergodic) probability
measures onto the space of $h$-invariant (respectively, $h$-ergodic) probability measures.

\item[(b)] If $\nu$ is an $h$-invariant probability measure in $N$ such that $\#\phi^{-1}(y)=1$
for $\nu$-almost every $y\in N$ then there exists a unique probability measure $\mu$ in $M$ such
that $\phi_*\mu=\nu$. This measure $\mu$ is $g$-invariant and it is $g$-ergodic if and only if
$\nu$ is $h$-ergodic.
\end{itemize}
\end{proposition}

\begin{proof}
It is easy to see from the relation $p\circ g = h \circ p$ that the push-forward of any $g$-invariant
probability measure $\mu$ is an $h$-invariant probability measure. To prove surjectivity, we need
the following consequences of the fact that $p$ is continuous:
\begin{enumerate}
\item $y \mapsto p^{-1}(y)$ is a map from $N$ to the space $\cK(M)$ of compact subsets of $M$;
\item this map is upper semicontinuous, with respect to the Hausdorff topology on $\cK(M)$.
\end{enumerate}
In particular, $y \mapsto p^{-1}(y)$ is measurable with respect to the Borel $\sigma$-algebras
of $N$ and $\cK(M)$. Then, \cite[Theorem~III.30]{CaV77} asserts that $p$ admits a measurable
section, that is, a measurable map $\sigma:N\to M$ such that $\sigma(y) \in p^{-1}(y)$ for
every $y\in N$ or, equivalently, $p\circ\sigma=\id$.

Given any $h$-invariant probability measure $\nu$, let $\omega=\sigma_*\nu$. Then $\omega$ is a
probability measure on $M$, not necessarily invariant, such that $p_*\omega=\nu$.
Since $p \circ g = h \circ p$, it follows that every iterate $f_*^j\omega$ also projects down to $\nu$.
Let $\mu_0$ be any accumulation point of the sequence
$$
\frac 1n \sum_{j=0}^{n-1} g_*^j\omega.
$$
It is well known that $\omega$ is $g$-invariant and, since the map $p_*$ is continuous,
it follows from the previous observations that $p_*\omega=\nu$. This proves surjectivity in the
space of invariant measures.

It is clear that $p_*\mu$ is $h$-ergodic whenever $\mu$ is $g$-ergodic. Conversely, let $\nu$ be
any $h$-ergodic probability measure. By the previous paragraph, there exists some $g$-invariant
probability measure $\mu$ such that $p_*\mu=\nu$. Let $\mu=\int \mu_P \, dP$ be the ergodic
decomposition of $\mu$ (see \cite[Chapter~5]{FET}). Since $p_*$ is a continuous linear operator,
\begin{equation}\label{eq.ergodic_decomposition}
\nu = p_*\mu = \int (p_*\mu_P) \, dP.
\end{equation}
By the previous observation, $p_*\mu_P$ is $h$-ergodic for almost every $P$. Thus, by uniqueness,
\eqref{eq.ergodic_decomposition} must be the ergodic decomposition of $\nu$. As we take $\nu$
to be ergodic, this implies that $p_*\mu_P=\nu$ for almost every $P$. That proves surjectivity
in the space of ergodic measures.

Now suppose that the set $Z=\{y\in N: \#\phi^{-1}(y)=1\}$ has full measure for $\nu$. Let $\mu$ be any
probability measure with $\phi_*\mu=\nu$. Then, in particular, $\mu(\phi^{-1}(Z))=1$.
Observe also that $\sigma\circ\phi(x)=x$ for every $x\in\phi^{-1}(Z)$. Then
$$
\mu = \sigma_*\phi_*\mu = \sigma_*\nu = \omega.
$$
This proves that $\mu$ is unique. By the surjectivity statements in the previous paragraph,
$\mu$ must be $g$-invariant and it must be $g$-ergodic if $\nu$ is $h$-ergodic.
\end{proof}

\begin{theorem}\label{main.C}
The map $\phi_*$ preserves the entropy and it is a bijection restricted to the subsets of invariant
ergodic probability measures with entropy larger than $\log\lambda_3$.
\end{theorem}

Ures~\cite{Ure12} proved a version of this result for measures of maximal entropy.

\begin{proof}
Let $\mu$ be any invariant probability measure. Clearly, $h_{\phi_*\mu}(A) \le h_\mu(f)$. On the
other hand, the Ledrappier-Walters formula~\cite{LeW77} implies that
$$
h_\mu(f) \le h_{\phi_*(\mu)}(A) + \max_{z\in\TT^3} h(f,\phi^{-1}(z)),
$$
where $h(f,K)$ denotes the topological entropy of $f$ on a compact set $K\subset\TT^3$.
See Viana, Oliveira~\cite[Section~10.1.2]{FET}, for instance. In our case, $K=\phi^{-1}(z)$ is a
one-dimensional segment whose images have bounded length. Hence, the topological entropy is
zero and so we get that $h_\mu(f) \le h_{\phi_*(\mu)}(A)$. This proves that $\phi$ preserves the
entropy.

Now, in view of Proposition~\ref{p.surjectivity}, we only have to show that if $\nu$ is an
$A$-ergodic probability measure with $h_\nu(A)>\log\lambda_3$ then $\phi^{-1}(y)$ consists
of a single point for $\nu$-almost every $y\in\TT^3$. Suppose otherwise. Then, using the
first part of Corollary~\ref{c.measurability}, there exists a positive measure set
$W\subset\TT^3$ such that $\phi^{-1}(y)$ is a non-trivial arc of a center leaf of $f$.
By ergodicity, we may suppose that $W$ has full measure. Let $L$ be any center leaf of $A$.
By Corollary~\ref{c.leaf_unique}, the pre-image $\phi^{-1}(W\cap L)$ is a subset of a unique
center leaf of $f$. Moreover, it is the union of the non-trivial arcs $\phi^{-1}(y)$ with
$y\in W \cap L$. Since these arcs are pairwise disjoint, there can only be
countably many of them. Thus, $W\cap L$ is countable. Then, we may use
Theorem~\ref{t.entropyestimation} to conclude that $h_\nu(A) \le \log \lambda_3$, which
contradicts the hypothesis.
\end{proof}

\begin{theorem}\label{main.D}
If $\mu$ is an ergodic measure of $f$ with negative center exponent, then $h_\mu(f)\le \log \lambda_3$
and there exists a full $\mu$-measure subset which intersects almost every center leaf on a single point.
\end{theorem}

\begin{proof}
Let us start with the following lemma:

\begin{lemma}\label{l.Kf}
There exists $K_f>0$ depending only on $f$ such that for any compact center segment $I$ there exists
$N_I\ge 1$ such the length of $f^{-n}(I)$ is bounded by $K_f$ for every $n \ge N_I$.
\end{lemma}

\begin{proof}
Let $\phi:\TT^3\to\TT^3$ be the semi-conjugacy introduced previously. Then $\phi(I)$ is a segment
inside a center leaf of $A$, and the same is true for the iterates:
$$
\phi(f^{-n}(I)) =  A^{-n}(\phi(I)) \quad\text{for every $n\ge 1$.}
$$
Since $A^{-1}$ contracts the center direction, because $\lambda_2>1$, the length of  $A^{-n}(\phi(I))$
goes to zero as $n\to+\infty$. As observed before, the map $\tphi$ is at bounded distance from the
identity. It follows that the distance between the endpoints of $f^n(I)$ in $\TT^3$ is uniformly
bounded when $n$ is large. Since the center leaves of $f$ are quasi-isometric (property (iii) above),
it follows that the length of the center segments $f^n(I)$ is uniformly bounded when $n$ is large,
as claimed.
\end{proof}

Let $\Gamma_\mu$ be the set of points $x\in\TT^3$ for which the center Lyapunov exponent is well
defined and coincides with the center Lyapunov exponent $\lambda^c(\mu)$ of the ergodic measure $\mu$.
Thus,
$$
\lim \frac 1n \log |Df^{-n} \mid E^c_f(x)| = - \lambda^c(\mu)
\quad\text{for every $x\in \Gamma_\mu$.}
$$
By ergodicity, $\Gamma_\mu$ is a full $\mu$-measure subset of the torus.

\begin{lemma}\label{l.deltamu}
There exists $\delta_\mu > 0$ such that for any $x \in \Gamma_\mu$ and any neighborhood $U$ of $x$ inside
the center leaf of $f$ that contains $x$, one has
$$
\liminf \frac 1n \sum_{i=0}^{n-1} \length(f^{-i}(U)) \ge \delta_\mu.
$$
\end{lemma}

A similar result was proven in  Lemma~3.8 in our previous paper~\cite{ViY13}. The present statement
is analogous, and even easier, because here we take the center direction to be one-dimensional.

\begin{corollary}
There exists $N_\mu \ge 1$ such that $\#(\Gamma_\mu \cap L) \le N_\mu$ for every center leaf $L$ of $f$.
\end{corollary}

\begin{proof}
Take $N_\mu = 3 K_f/\delta_\mu$. Suppose that $\#(\Gamma_\mu \cap L) > N_\mu$ for some center
leaf $L$. Fix pairwise disjoint neighborhoods around each of these points, and let $I\subset L$ be a
compact segment containing these neighborhoods. From Lemma~\ref{l.deltamu} , we get that
$$
\frac 1n \sum_{i=0}^{n-1} \length(f^{-i}(I)) > N_\mu \frac{\delta_\mu}{2} > K_f
$$
for every large $n$, which contradicts Lemma~\ref{l.Kf}.
\end{proof}

We are left to prove that, up to replacing $\Gamma_\mu$ by some full measure invariant subset, we may take $N_\mu=1$.
This can be seen as follows. Fix an orientation of the leaves of $f$ once and for all (it is clear that the
center foliation of $A$ is orientable and then we may use the semi-conjugacy $\phi$ to define an orientation
of the center leaves of $f$). Let $\Gamma_{\min}$ be the subset of $\Gamma$ formed by the first points of
$\Gamma_\mu$ on each center leaf, with respect to the chose orientation. It is clear that $\Gamma_{\min}$ is
invariant under $f$, because $\Gamma_\mu$ is. We claim that $\Gamma_{\min}$ is a measurable set. Let us assume
this fact for a while. If $\Gamma_{\min}$ has positive measure then, by ergodicity, it has full measure.
Since $\Gamma_{\min}$ intersects every leaf in at most one point, this proves that we may indeed take $N_\mu=1$.
If $\Gamma_{\min}$ has zero measure, just replace $\Gamma_\mu$ with $\Gamma_\mu\setminus \Gamma_{\min}$ and
start all over again. Notice that this $N_\mu$ is replaced with $N_\mu-1$ and so this argument must stop in
less than $N_\mu$ steps.

It remains to check that $\Gamma_{\min}$ is indeed a measurable set. We need

\begin{lemma}\label{l.diameter}
There exists $R>0$ such that the diameter of $\Gamma_\mu \cap L$ inside every center leaf $L$ is
less than $R$.
\end{lemma}

\begin{proof}
By Proposition~\ref{p.surjectivity}(a), the projection $\phi_*\mu$ is an ergodic measure for $A$.
Keep in mind that, by Corollary~\ref{c.leaf_unique}, $\phi$ maps center leaves to center leaves, in an one-to-one fashion.
Thus, $\phi(\Gamma_\mu)$ is a full measure that intersects each center leaf of $A$ at finitely many points.
By Proposition~\ref{p.atomic}, it follows that the intersection consists of a single point.
In other words, the intersection $\Gamma_\mu \cap L$ with each center leaf $L$ is contained in $\phi^{-1}(z)$
for some $z\in \phi (L)$. Then the conclusion of the lemma follows directly from Proposition~\ref{p.semiconjugation}.
\end{proof}

Given $r>0$ and any disk $D$ transverse to the center foliation, let $D(r)$ denote the union of the center
segments of radius $r$ around the points of $D$.
Consider a finite family $\{D_i: i=1, \dots, l\}$ of (small) disks transverse to the center foliation, such that
\begin{enumerate}
\item[(i)] each $D_i(R+2)$ is homeomorphic to the product $D_i \times [-R-2,R+2]$;
\item[(ii)] the sets $D_i(1)$, $i=1, \dots, l$ cover $M$.
\end{enumerate}

For each $i=1, \dots, l$, denote by $\Gamma_{\min}(i)$ the set formed by the first point of $\Gamma$ in the
center leaf through each point in $\Gamma \cap D_i(1)$. Notice that
$$
\Gamma_{\min}(i) \subset D_i(R+2) \quand \Gamma_{\min} = \cup_i \Gamma_{\min}(i)
$$
as a consequence of Lemma~\ref{l.diameter}. Thus we only have to check that, up to replacing $\Gamma_\mu$ by some
invariant full measure subset, each $\Gamma_{\min}(i)$ is a measurable set. The latter can be seen as follows.

Let $i$ be fixed. Identify $D_i(R+2)$ with $D_i\times[-R-2,R+2]$ through the homeomorphism in condition (i) above.
Let $E\subset D_i$ be the vertical projection of $\Gamma \cap D_i(1)$. Theorem~III.23 in \cite{CaV77} ensures that
$E$ is a measurable subset of $D_i$. Moreover, $\Gamma_{\min}(i)$ is the graph of a function $\sigma:E \to [-R-2,R+2]$.
Our goal is to prove that this function is measurable.
If $\Gamma_\mu$ is compact then the function $\sigma$ is lower semi-continuous and thus measurable. In general,
by Lusin, we may find an increasing sequence of compact sets $\Gamma_k\subset\Gamma_\mu$ such that their union
$\Gamma'(i)$ has full measure in $\Gamma_\mu \cap D_i(1)$. By the previous observation, the function
$\sigma_k:E_k\to [-R-2,R+2]$ associated with each $\Gamma_k$ is measurable. The function $\sigma':E' \to [-R-2,R+2]$
associated with $\Gamma'(i)$ is given by
$$
E'=\cup_k E_k \quand \sigma'(z) = \inf_k \sigma_k(z).
$$
Thus, $\sigma'$ is a measurable function as well. To get the claim, just replace $\Gamma_\mu$ with the invariant subset
obtained by removing the orbits through all zero measure sets $\Gamma_\mu \cap D_i(1) \setminus \Gamma'(i)$.

This completes the proof of Theorem~\ref{main.D}.
\end{proof}

\begin{remark}
In the context of Theorem~\ref{main.D}, the map $\phi_*$ is not injective at $\mu$: there is at least one more ergodic
measure $\nu$ such that $\phi_*\mu=\phi_*\nu$. This can be seen as follows. The assumption that the center Lyapunov
exponent of $\mu$ is negative implies that $\phi^{-1}(z)$ is a non-trivial segment for $\phi_*\mu$-typical points $z$.
Let $x$ be an endpoint of $\phi^{-1}(z)$, for any such $z$, and $\nu$ be any accumulation point of the time average
over the orbit of $x$. Then $\phi_*\nu=\phi_*\mu$, because $z$ is taken to be $\phi_*\mu$ typical. Moreover, the center
Lyapunov exponent of $\nu$ can not be negative, for otherwise there would a neighborhood of $x$ inside $\phi^{-1}(z)$,
which would contradict the choice of $x$.
\end{remark}

\subsection{Proof of Theorem~\ref{main.criterion}}

By Theorem~\ref{main.C}, the measure $\nu=\phi_*\mu$ is ergodic and has the same entropy as $\mu$.
In particular, $h_\nu(A) > \log \lambda_3$. Let $Z\subset M$ be any full $\mu$-measure set
and $Z'=Z \cap \phi^{-1}(Y)$, where $Y$ is the as in Corollary~\ref{c.measurability}.
Then $\phi(Z')$ is a measurable subset of $\TT^3$.
Moreover, $\phi(Z')$ has full $\nu$-measure, because $Y$ has full $\nu$-measure (this is contained in the second part
of the proof of Theorem~\ref{main.C}) and so $Z'$ has full $\mu$-measure.
Thus, by Theorem~\ref{t.entropyestimation}, $\phi(Z')$ intersects almost every center leaf of $A$ on an uncountable
subset. By Corollary~\ref{c.leaf_unique}, the pre-image of every center of $A$ is a center leaf of $f$. It follows
that $Z'$ intersects almost every center leaf on an uncountable subset. Then the same holds for $Z$, of course.

We have seen in Theorem~\ref{main.D} that if the center Lyapunov exponent is negative then some full measure subset
intersects almost every center leaf in a single point. In view of the previous paragraph, this ensures that in the
present situation the center exponent is non-negative. We are left to show that the center Lyapunov exponent is
actually positive when $f$ is a $C^2$ diffeomorphism.

By contradiction, assume that the center exponent of $f$ is non-positive. Then the strong-unstable leaf $\cF^{uu}$
coincides with the Pesin unstable manifold at $\mu$-almost every point. Define the \emph{exponential volume growth rate}
of any disk $D$ contained in some strong-unstable leaf of $f$ to be
\begin{equation}\label{eq.growthrate}
G(D)=\liminf_{n\to \infty} \frac{1}{n}\log \frac{\Leb(f^n(D))}{\Leb(D)}.
\end{equation}
It was shown by Cogswell~\cite{Cog00} that $h_\mu(f)\leq G(D)$ whenever $f$ is $C^2$ and $D$ is a neighborhood
of a $\mu$-typical point $x\in M$ inside its Pesin unstable manifold.

Clearly, if the center Lyapunov exponents are non-positive then the Pesin unstable manifold coincides with the
strong-unstable leaf through the point.
So, to complete the proof of Theorem~\ref{main.criterion} it suffices to show that $G(D) \leq \log \lambda_3$ for
any segment $D$ inside some strong-unstable leaf. This can be seen as follows. Let $\tx_1$ and $\tx_2$
be the endpoints of some lift $\tilde D$ of the segment $D$ to the universal cover. By \eqref{eq.closetoidentity},
$$
\|\tf^n(\tx_1)-\tf^n(\tx_2)\|
\le 2C + \|A^n(\tphi(\tx_1))-A^n(\tphi(\tx_2))\|
= 2C + \lambda_3^n \|\tphi(\tx_1) - \tphi(\tx_2)\|.
$$
It has been show by Potrie~\cite[Corollary~7.7]{Pot15} that in the present setting the lift of the unstable foliation
to the universal is quasi-isometric: the distance between any two points along a leaf is bounded by some affine
function of the distance of the two points in the ambient space. Thus, in particular, there exists a uniform constant
$Q >0$ such that
$$
|\tf^n(\tilde{D})|
\leq Q + Q \|\tf^n(\tx_1)-\tf^n(\tx_2)\|
\leq Q (2C+1) + Q \lambda_3^n \|\tphi(\tx_1) - \tphi(\tx_2)\|.
$$
Replacing this estimate in the definition of $G(D)$ we get that $G(D) \le \log\lambda_3$, as claimed.

The proof of Theorem~\ref{main.criterion} is complete.

\subsection{Proof of Corollary~\ref{main.example}}

By Hammerlindl, Ures~\cite[Theorem~7.2]{HaU14}, every $C^2$ volume-preserving partially hyperbolic diffeomorphism
$g\in \cD(A)$ whose integrated center Lyapunov exponent $\lambda^c(g)$ is different from zero is ergodic.
Thus, since the map $g \mapsto \lambda^c(g)$ is continuous, every volume-preserving $g$ in a neighborhood of $f$
is ergodic. We begin by claiming that the disintegration of volume along the center foliation of such a $g$
cannot be atomic.

Indeed, suppose that the disintegration is atomic. Let $\cB_1, \dots, \cB_k$ be a finite cover of $M$ by foliation
charts such that the conditional probabilities of each $\Leb\mid \cB_i$ are purely atomic for almost every plaque.
Equivalently (Appendix~\ref{app.atomic}), every $\cB_i$ admits a full measure subset $Z_i$ whose intersection with
every plaque is countable. Then $Z=Z_1 \cup \cdots \cup Z_k$ is a full measure subset of $M$ that intersects every
center leaf on a countable set. This contradicts Theorem~\ref{main.criterion}.

Now we prove that conditional probabilities of $\Leb$ along center leaves cannot be absolutely continuous respect
to Lebesgue measure. Let $\Gamma$ be the set of points $x\in M$ such that
$$
\lim_n \frac{1}{n}\log \|Dg^n \mid E^c_f(x)\|=\lambda^c(g),
$$
By ergodicity and the Birkhoff theorem, $\Gamma$ has full volume.  Fix $\vep < (\lambda^c(g)-\log \lambda_2)/2$.
Then, there exists a measurable function $n(x): \Gamma\to \mathbb{N}$ such that
$$
\|Dg^n \mid E^c_g(x)\| \geq e^{\vep n}\lambda_2^n
\quad\text{for any $x\in \Gamma$ and $n\geq n(x)$.}
$$
Take $n_0\ge 1$ sufficiently large, such that the set $\Gamma_0=\{x\in \Gamma: n(x) \leq n_0\}$ has positive volume.
Assuming, by contradiction, that the disintegration of the volume measure along the center foliation is
absolutely continuous, we get that there exists some center plaque $L$ such that $\Gamma_0 \cap L$ has positive
Lebesgue measure. By the definition of $\Gamma$,
$$
|g^n(L)| \geq \Leb_{g^n(L)}(g^n(\Gamma_0)) \geq e^{\vep n}\lambda_2^n \Leb_L(\Gamma_0)
\quad\text{for any $n\geq n_0$,}
$$
which implies that $G(L) \geq \log\lambda_2+\vep$.
Thus, to reach a contradiction, it suffices to show that $G(L) \leq \log\lambda_2$.

For proving this latter claim, we use a variation of an argument in the proof of Theorem~\ref{main.criterion}.
Let $\tx_1$ and $\tx_2$ be the endpoints of some lift $\tilde L$ of the segment $L$ to the universal cover.
By Proposition~\ref{p.center-to-center}, $\tphi(\tx_1)$ and $\tphi(\tx_2)$ belong to the
same center leaf of $A$. Then,
$$
\|\tphi(\tg^n(\tx_1)) - \tphi(\tg^n(\tx_2))\|
= \|A^n(\tphi(\tx_1)) - A^n(\tphi(\tx_1))\|
= \lambda_2^n \|\tphi(\tx_1) - \tphi(\tx_2)\|
$$
for every $n\ge 1$. Since $\tphi$ is uniformly close to the identity, by \eqref{eq.closetoidentity},
it follows that
$$
\|\tg^n(\tx_1) - \tg^n(\tx_2)\|
\leq \lambda_2^n\|\tphi(\tx_1) - \tphi(\tx_2)\| + 2C.
$$
Then, using the quasi-isometry property of the center foliation (recall (iii) in the
proof of Proposition~\ref{p.semiconjugation}),
$$
|\tg^n(L)|
\le Q \|\tg^n(\tx_1) - \tg^n(\tx_2)\| + Q
\le Q\lambda_2^n\|\tphi(\tx_1) - \tphi(\tx_2)\| + Q(2C+1),
$$
where $Q$ is a uniform constant. Replacing this in \eqref{eq.growthrate} we get
that $G(L) \le \log\lambda_2$, as claimed.

\section{Upper absolute continuity}

We have seen previously that the disintegration of Lebesgue measure along center leaves may be \emph{singular without
being atomic}. This adds to the previously known types of behaviour for the center foliation: \emph{Lebesgue disintegration}
(i.e. \emph{leafwise absolute continuity}) and \emph{atomic disintegration}.

In this section we want to refine our understanding of the non-singular case. Let $\Leb$ denote the Lebesgue measure
in the ambient manifold and $\Leb_L$ be the Lebesgue measure restricted to some submanifold $L$.
Following~\cite{AVW11,AVW12}, we say that a foliation $\cF$ is \emph{upper leafwise absolutely continuous} if
$\Leb_L(Y ) = 0$ for every leaf $L$ through a full Lebesgue measure subset of points $z \in M$ implies $\Leb(Y) = 0$.
Similarly, $\cF$ is \emph{lower leafwise absolutely continuous} if for every zero $\Leb$-measure set $Y \subset M$ and
$\Leb$-almost every $z\in M$, the leaf $L$ through $z$ meets $Y$ in a zero $\Leb_L$-measure set.
Thus, the foliation is upper leafwise absolutely continuous if the conditional measure of $\Leb$ along a typical leaf $L$ is
absolutely continuous with respect to $\Leb_L$ and it is lower leafwise absolutely continuous if $\Leb_L$ is absolutely
continuous with the respect to the conditional measure of $\Leb$ along a typical leaf $L$.
So, Lebesgue disintegration (leafwise absolute continuity) is the same as both upper and lower leafwise absolute continuity.

Pesin theory may be used to show that upper leafwise absolute continuity is actually quite common (see \cite[Proposition~6.2]{ViY13}
for a precise statement). Here we describe fairly robust examples whose center foliations are upper but not lower leafwise
absolutely continuous.

We start from a construction due to Kan~\cite{Ka94}. Let $f_0:S^1 \times [0,1] \to S^1 \times [0,1]$ be a
$C^2$ map of the cylinder of the form  $f_0(\theta,t)=(3\theta, h_\theta(t))$ with
\begin{enumerate}
\item $h_\theta(i)=i$ for every $i\in\{0, 1\}$ and every $\theta\in S^1$;

\item $|h_\theta'(t)| \le c < 3$ for some $c$ and every $\theta\in S^1$;

\item $\int \log |h'_\theta(i)| \, d\theta <0$ for $i\in\{0,1\}$;

\item $|h_0'(0)| < 1$ and $|h_{1/2}'(1)| < 1$ and $h_0(t) < t < h_{1/2}(t)$ for $t\in (0,1)$.
\end{enumerate}

The first condition means that $f_0$ preserves the two boundary components of the cylinder $S^1 \times [0,1]$.
The second one ensures that $f_0$ is a partially hyperbolic endomorphism of the cylinder, with the vertical segments
as center leaves. The restriction of $f_0$ to each boundary component $S^1\times\{i\}$  is uniformly expanding and
preserves the Lebesgue measure $m_i$ on the boundary component.
The third condition means that, for either boundary component, the transverse Lyapunov exponent is negative.
Finally, the fourth condition means that $0$ is an attractor for $h_0$ and $1$ is an attractor for $h_{1/2}$ and
their basins contain the interval $(0,1)$.

Let $\cK$ be a small neighborhood of $f_0$ inside the space of $C^2$ maps of the cylinder preserving both boundary
components. Every $f\in\cK$ is partially hyperbolic, with almost vertical center foliation $\cF^c_f$, and admits
absolutely continuous ergodic invariant measures $m_{i,f}$ on the boundary components. These measures vary continuously
with the map and so their center Lyapunov exponents
$$
\int \log |Df \mid E^c_f(\theta,i)| \, dm_{i,f}(\theta) \approx \int \log |h'_\theta(i)| \, d\theta
$$
($E^c_f$ denotes the center bundle, tangent to $\cF^c_f$) are still negative. This ensures that both $m_{0,f}$
and $m_{1,f}$ are physical measures for $f$. As observed by Kan~\cite{Ka94}, the basins $B(m_{i,f})$ are
intermingled - they are both dense - and their union has full measure in the ambient cylinder. Denote by
$p_0(f)$ and $p_1(f)$ the continuations of the fixed saddle-points $(0,0)$ and $(1/2,1)$, respectively.

\begin{theorem}\label{t.kantype}
For any $f\in \cK$ such that $\partial_\theta f(p_0) \neq \partial_\theta f(p_1)$, the center foliation is upper
leafwise absolutely continuous but not lower leafwise absolutely continuous.
\end{theorem}

\begin{proof}
Let $\pi_f$ be the holonomy map of the center foliation of $f$ from the bottom boundary component $S^1 \times \{0\}$
to the top boundary component $S^1 \times \{1\}$. Then $\pi_f$ is a homeomorphism and the fact that the center foliation
is invariant means that it conjugates the restrictions of $f$ to the two boundary components.
It is well-known that a conjugacy between two $C^2$ expanding maps is either $C^1$ or completely singular.
The assumption $\partial_\theta f(p_0) \neq \partial_\theta f(p_1)$ prevents the former possibility,
since $\pi_f$ maps $p_{0,f}$ to $p_{1,f}$. Thus, $\pi_f$ is completely singular and so the measures $m_{1,f}$
and $m_{1,f}^*= (\pi_f)_*m_{0,f}$ are mutually singular. Similarly, the measures $m_{0,f}$ and $m_{0,f}^*= (\pi_f)_*m_{1,f}$
are mutually singular. In other words, for $i=0, 1$ there exists a full $m_{i,f}$-measure subset $\Lambda_{i,f}$ of
$S^1\times\{i\}$ such that the sets of center leaves through $\Lambda_{0,f}$ and $\Lambda_{1,f}$ are disjoint.

We claim that all four invariant measures $m_{i,f}$ and $m_{i,f}^*$, $i=0, 1$ have negative center exponents,
assuming $f$ is close enough to $f_0$.
This can be seen as follows. First of all, that is true for $f=f_0$ and in this case, $m_{i,f}=m_{i,f}^* = $
Lebesgue measure along $S^1\times\{i\}$. Then, observe that $f\mapsto m_{i,f}$ is continuous (because the
absolutely continuous invariant measure of a $C^2$ expanding map depends continuously on the map) and $\pi_f$
also depends continuously on $f$ (note that $f=f_0$ the holonomy map is just $(\theta,0) \mapsto (\theta,1)$).
Thus, all these measures vary continuously with $f$. Since the center bundle is one-dimensional, so
do their center Lyapunov exponents. Our claim follows.

We also need the following fact:

\begin{lemma}\label{l.basin_stable_set}
Up to Lebesgue measure zero, for $i=0, 1$, the basin of $m_{i,f}$ coincides with the union of the Pesin stable
manifolds of the points in $\Lambda_{i,f}$, which is contained in the union $W^c(\Lambda_{i,f})$ of the center
leaves through the points of $\Lambda_{i,f}$.
\end{lemma}

This follows from a standard density point argument, see for instance \cite[Proposition~11.1]{Beyond}.
Similar statements appeared also in our previous papers \cite[Lemma~4.6]{DVY16} and \cite[Proposition~6.9]{ViY13},
in somewhat different situations.

Denote $Y_f = W^c(\Lambda_{0,f}) \setminus B(m_{0,f}) \cup W^c(\Lambda_{1,f}) \setminus B(m_{1,f})$.
On the one hand, since the union of the basins $B(m_{i,f})$, $i=0, 1$ has full Lebesgue measure (Kan~\cite{Ka94}),
the set $Y_f$ has zero Lebesgue measure in $S^1\times[0,1]$ and Lebesgue almost every center leaf is contained in
$W^c(\Lambda_{0,f}) \cup W^c(\Lambda_{1,f})$.
On the other hand, $Y_f$ contains the Pesin stable manifold of $m_{i,f}^*$-almost every point, for $i=0, 1$.
In particular, the intersection of $Y_f$ with Lebesgue almost every center leaf contains a non-trivial segment
and, thus, has positive Lebesgue measure inside the center leaf.
This proves that the center foliation is not lower leafwise absolutely continuous.

Finally, by Lemma~\ref{l.basin_stable_set}, the union of the Pesin stable manifolds through $\Lambda_{0,f} \cup \Lambda_{1,f}$
has full Lebesgue measure in $S^1\times[0,1]$. Thus, the partition of the ambient into center leaves coincides
almost everywhere with the Pesin stable lamination. By Pesin theory (see~\cite{Pes76,PSh89} the latter is
absolutely continuous. This ensures that the disintegration of the Lebesgue measure is absolutely continuous
along Pesin stable manifolds and, consequently, along center leaves. This proves that the center foliation is upper
leafwise absolutely continuous.
\end{proof}

\appendix

\section{Atomic disintegration}\label{app.atomic}

Denote by $D^k$ the closed unit disk in $\RR^k$. Let $\cF$ be a \emph{foliation} of dimension $k\ge 1$ of some
manifold $M$ of dimension $d>k$.
By this we mean that every point of $M$ is contained in the interior of some \emph{foliation box}, that is, some
image $\cB$ of a topological embedding
$$
\Phi:D^{d-k} \times D^{k} \to M
$$
such that every \emph{plaque} $P_x=\Phi(\{x\} \times D^{k})$ is contained in a leaf of $\cF$.
We say that $\cF$ has \emph{$C^r$ leaves} if every
$$
\Phi(x,\cdot):D^k \to M, \qquad y \mapsto \Phi(x,y)
$$
is a $C^r$ embedding depending continuously on $x$ in the $C^r$ topology.

Let $\cB$ be a foliation box, identified with the product $D^{d-k}\times D^k$ through the corresponding
homeomorphism $\cB$. By Rokhlin's disintegration theorem (see~\cite[Theorem~5.1.11]{FET}) there exists a
probability measure $\hat\nu$ on $D^{d-k}$ and a family of probability measures $\{\nu_x: x\in D^{d-k}\}$
such that
\begin{equation}\label{eq.rokhlin}
\nu(E) = \int_{D^{d-k}} \nu_x\big(\{y\in D^{k}: (x,y) \in E\}\big) \, d\hat\nu(x)
\end{equation}
for every measurable set $E\subset\cB$. In fact, $\hat\nu$ is just the projection of $\nu$ on the first
coordinate and the family $\{\nu_x: x\in D^{d-k}\}$ is essentially uniquely determined.

We say that $\nu$ has \emph{atomic disintegration} along $\cF$ if for every foliation box $\cB$ and
$\nu$-almost every $x\in\cB$ the \emph{conditional measure} $\nu_x$ gives full weight to some countable
set (equivalently, $\nu_x$ is a countable linear combination of Dirac masses).

\begin{lemma}\label{l.atomic}
$\nu$ has atomic disintegration if and only if for every foliation box $\cB$ there exists a full $\nu$-measure set
$Z\subset\cB$ whose intersection with $\hat\nu$-almost every plaque $P_x$ is countable (possibly finite).
\end{lemma}

\begin{proof}
Suppose that there exists some full $\nu$-measure set $Z\subset\cB$ whose intersection with $\hat\nu$-almost
every plaque $P_x$ is countable. By \eqref{eq.rokhlin}, it follows that $\nu_x(Z \cap P_x) = 1$
(and so $\nu_x$ is a purely atomic measure) for $\hat\nu$-almost every $x$.
The converse is also true: if $\nu_x$ is purely atomic for $\hat\nu$-almost every $x$ then one may find
a full measure subset $Z$ of $\cB$ that intersects every plaque on a countable subset.
This can be deduced from the claim in Rokhlin~\cite[\S~1.10]{Rok67a} but, for the reader's convenience,
we provide a quick direct explanation.

The idea is quite simple: we take $Z = \cup_x \{x\} \times Y(x)$ where each $\{x\} \times Y(x)$
is a countable full $\nu_x$-measure subset of the plaque $P_x$. The main point is to check that $Z$ is
a measurable set (up to measure zero). Once that is done, \eqref{eq.rokhlin} immediately gives that
$Z$ has full $\nu$-measure. To prove measurability, start by fixing some countable basis $\cV$ for
the topology of $\cB$. It is also part of Rokhlin's theorem that the map $x\mapsto \nu_x(V)$ is measurable
(up to measure zero) for any measurable set $V\subset\cB$. Thus, given any $\vep>0$, one may find a
compact set $K_\vep\subset D^{d-k}$ such that $\hat\nu(K_\vep) > 1 - \vep$ and $x \mapsto \nu_x(V)$ is
continuous on $K_\vep$, for every $V\in\cV$. In particular, $x \mapsto \nu_x$ is continuous with respect
to the weak$^*$ topology for $x\in K_\vep$.

It is no restriction to suppose that $\nu_x$ is purely atomic for every $x\in K_\vep$, and we do so.
Let $\vep>0$ be fixed. It is clear that, given any $\delta>0$, the set
$$
\Gamma_\delta(x) = \{y \in D^{k}: \nu_x(\{y\})\ge\delta\}
$$
is finite, and hence compact. Moreover, the properties of $K_\vep$ ensure that the function
$x\mapsto \Gamma(x,\delta)$ is upper semi-continuous on $x\in K_\vep$. In other words,
$$
\Lambda(\vep,\delta)=\{(x,y): x\in K_\vep \text{ and }  y \in \Gamma_\delta(x)\}
$$
is a closed subset of $\cB$. Since $\nu_x$ is purely atomic for every $x\in K_\vep$, the union
$\Lambda(\vep)=\cup_n \Lambda(\vep,1/n)$ is a (measurable) full measure subset of $K_\vep \times D^{k}$
contained in $Z$. Then, $\cup_m \Lambda(1/m)$ is a (measurable) full measure subset of $\cB$ contained
in $Z$. This proves that $Z$ is measurable up to measure zero, as claimed.
\end{proof}

Notice that the statement of the lemma concerns the intersection of $Z$ with plaques of the foliation,
not entire leaves. In the special case of foliations of dimension $k=1$ one can do a bit better.
Indeed, consider any finite cover of the ambient manifold by foliation boxes. Since $1$-dimensional
manifolds have only two ends, every leaf can intersect these foliation boxes at most countably many
times, that is, every leaf is covered by countably many plaques. Thus for $k=1$ the condition in
Lemma~\ref{l.atomic} may be reformulated equivalently as follows: there exists a full $\nu$-measure
set $Z\subset\cB$ such that $Z \cap \cF_x$  is countable for $\hat\nu$-almost every $x$.
This conclusion extends to large $k$ under the additional condition that every leaf has countably
many ends.

\end{document}